\documentclass[ejs]{imsart}

\usepackage{amsthm,amsmath}
\RequirePackage[numbers]{natbib}
\RequirePackage[colorlinks,citecolor=blue,urlcolor=blue]{hyperref}
\usepackage{latexsym,amsfonts,amsmath,amssymb,mathrsfs}
\usepackage{url}
\usepackage{graphicx}
\allowdisplaybreaks
\usepackage{color}
\definecolor{darkred}{RGB}{139,0,0}
\definecolor{darkgreen}{RGB}{0,100,0}
\definecolor{darkmagenta}{RGB}{139,0,139}
\definecolor{darkpurple}{RGB}{110,0,180}
\definecolor{darkblue}{RGB}{40,0,200}
\definecolor{darkorange}{RGB}{255,140,0}

\usepackage{tikz}
\usetikzlibrary{shapes.geometric, arrows}
\usepackage{euscript}
\usepackage{verbatim}
\usepackage{hyperref}
\usepackage{enumerate}
\doi{10.1214/154957804100000000}
\pubyear{0000}
\volume{0}
\firstpage{0}
\lastpage{0}
\startlocaldefs
\newtheorem{theorem}{Theorem}
\newtheorem{lemma}{Lemma}
\newtheorem{algorithm}{Algorithm}
\newtheorem{remark}{Remark}

\newtheorem{definition}{Definition}
\newtheorem{corollary}{Corollary}

\newtheorem{example}{Example}
\newtheorem{assumption}{Assumption}
\newcommand{\abs}[1]{\left\vert #1 \right\vert}
\newcommand{\bsx}{\boldsymbol{x}}

\endlocaldefs

\begin{document}

\begin{frontmatter}

% "Title of the Paper"
\title{Discrepancy Bounds for  Deterministic Acceptance-Rejection Samplers\thanksref{t1}}
\thankstext{t1}{H. Zhu was supported by a PhD scholarship from the University of New South Wales. J. Dick was supported by a Queen Elizabeth 2 Fellowship from the Australian Research Council. Helpful comments by Art Owen and Su Chen are gratefully acknowledged.}
\runtitle{Discrepancy for  Deterministic Acceptance-Rejection Samplers}

% indicate corresponding author with \corref{}
 %\author{\fnms{John} \snm{Smith}\thanksref{t2}\corref{}\ead[label=e1]{smith@foo.com}\ead[label=e2,url]{www.foo.com}}
%%\thankstext{t2}{Thanks to somebody}
% \printead{houying.zhu@student.unsw.edu.au}\\ \printead{josef.dick@unsw.edu.au}}

\author{\fnms{Houying} \snm{Zhu}\ead[label=e1]{houying.zhu@student.unsw.edu.au;josef.dick@unsw.edu.au}}
\and
\author{\fnms{Josef} \snm{Dick}\ead[label=e2]{josef.dick@unsw.edu.au}}
\address{School of Mathematics and Statistics\\ The University of New South Wales, Sydney, Australia \\}
\address{\printead{e1}}
\runauthor{Zhu and Dick}

\begin{abstract}
We consider  an acceptance-rejection sampler based on a  deterministic driver sequence. The deterministic sequence is chosen such that the discrepancy between the empirical target distribution and the target distribution
 is small. We use quasi-Monte Carlo (QMC) point sets for this purpose. The empirical evidence shows
convergence rates beyond the crude Monte Carlo rate of $N^{-1/2}$.
 We prove that the discrepancy of samples
  generated by the  QMC acceptance-rejection sampler is bounded from above by $N^{-1/s}$. A lower bound shows that for any given driver sequence, there always exists a target density such that the star discrepancy is at most $N^{-2/(s+1)}$.
   %and converges at a rate of $N^{-1/s}$ for a given target density defined on $[0,1]^{s-1}$.
For a general density, whose domain is the real state space $\mathbb{R}^{s-1}$, the inverse  Rosenblatt transformation can be used to convert samples
from the $(s-1)-$dimensional cube to $\mathbb{R}^{s-1}$. We show that this transformation is measure preserving.
This way, under certain conditions, we obtain the same convergence rate for a general target density defined in $\mathbb{R}^{s-1}$. Moreover,   we also consider  a deterministic reduced acceptance-rejection algorithm recently introduced by Barekat and Caflisch [F. Barekat and R.Caflisch. Simulation with Fluctuation and Singular Rates. ArXiv:1310.4555[math.NA], 2013.]
\end{abstract}



% history:
% \received{\smonth{1} \syear{0000}}

%\tableofcontents

\end{frontmatter}

%\tableofcontents

\section{Introduction}

The Monte Carlo (MC) method is one of  the widely used numerical methods for simulating probability distributions. However, sometimes it is not possible to sample from a given target distribution.
Markov chain Monte Carlo (MCMC)  methods have been developed to address this problem. Instead of sampling independent points directly, MCMC samples
from a Markov chain whose limiting distribution is the target distribution.
 MCMC has widened the applications of MC in many different fields \cite{Caflisch1998,MeynTweedie1993}. Another deficiency of MC {algorithms} is its slow convergence rate.
Quasi-Monte Carlo (QMC) algorithms  on the other hand {perform} better in improving the convergence rate of Monte Carlo which partially depends on generating  samples with small discrepancy. For a survey of QMC  we refer to \cite{JFS2013,Josef}. Putting the QMC idea into MCMC is a good way to  improve the convergence rate and widen practical applications.  Recently many  results in this direction have been achieved \cite{ChenThesis2011,CDO2011,TribbleThesis2007,TribbleOwen2005}. With this paper we add another result in this direction by using a deterministic driver {sequence} in an acceptance-rejection algorithm. We prove discrepancy bounds of order $N^{-1/s}$, where the dimension of the state space is $s-1$ and $N$ is the number of samples. The discrepancy here is a generalization of the concept of the Kolmogorov-Smirnov test between the empirical distribution of the samples and  the target distribution to higher dimension.
A more detailed description of our results will be provided below.

\subsection{Previous work on MCMC and QMC}
In the following we describe some previous research on  QMC and MC. L'Ecuyer studied the convergence behavior of randomized quasi-Monte Carlo for discrete-time  Markov chains in \cite{Pierre2008}, known as array-RQMC. The general idea is to obtain a better approximation of the target distribution than with the plain   Monte Carlo method by using randomized QMC. In a different direction, Tribble~\cite{TribbleThesis2007} and Tribble and Owen~\cite{TribbleOwen2005}
  established {a} condition under  which low discrepancy sequences can be used for consistent MCMC estimation for finite state spaces.  It has been shown that replacing an IID sequence by a completely uniformly distributed sequence also implies a consistent estimation in finite state spaces.   A construction of  weakly completely uniformly distributed sequences is also proposed in \cite{TribbleOwen2005}. As a sequel to the work of Tribble, Chen in   his thesis \cite{ChenThesis2011} and Chen, Dick and Owen~\cite{CDO2011} demonstrated that Markov chain quasi-Monte Carlo (MCQMC) algorithms using a completely uniformly distributed sequence as
  driver sequence gives a consistent result under certain assumptions on the update function and Markov chain. Further, Chen~\cite{ChenThesis2011} also showed that MCQMC can achieve
    a  convergence rate of $O(N^{-1+\delta})$ for any $\delta>0$ under certain conditions,  but he only showed the existence of a driver sequence.

In our recent work \cite{DRZ2013}, done with  Rudolf,  we prove upper bounds on the discrepancy under the assumptions that the Markov chain is uniformly ergodic
and the driver sequence is deterministic rather than independent uniformly distributed random variables. In particular, we show the existence of driver sequences for which the discrepancy of the Markov chain from the target distribution with respect
to certain test sets converges with (almost) the usual Monte Carlo rate of $N^{-1/2}$. A drawback of this result is that we
 are currently  not able to give an explicit construction of a driver sequence for  which our discrepancy bounds hold for uniformly ergodic Markov chains. Garber and Choppin in \cite{GC2014} adapted  low discrepancy point sets instead of random numbers in sequential Monte Carlo (SMC). They proposed  a new algorithm named sequential quasi-Monte Carlo (SQMC).
  They constructed consistency and stochastic bounds based on randomized QMC point set for this algorithm. It is an open problem to obtain  deterministic bounds for SQMC.
More literature review about applying QMC to MCMC problems can be found in \cite[Section 1]{CDO2011}.

{\subsection{Acceptance-rejection algorithms}}
We now give a description of the algorithms in this paper.
Let $\psi: D\to{\mathbb{R}}_{+}=[0,\infty)$ be our target density function, where $D\subseteq \mathbb{R}^{s-1}$ and $\mathbb{R}_{+}=[0,\infty)$. We consider the cases where $D=[0,1]^{s-1}$ or $D=\mathbb{R}^{s-1}$.
Assume that it is not possible to sample directly from the target distribution.
One possible solution to obtain samples from $\psi$  is to choose a proposal density $H$ from which we can sample and then use an acceptance-rejection algorithm. Assume there exists a constant $L<\infty$ such that $\psi(\boldsymbol{x}) < L H(\boldsymbol{x})$ for all $\boldsymbol{x}\in D$. The  following algorithm can be used to obtain samples with distribution $\psi$.

\begin{algorithm}\label{Basic_AR}(Random acceptance-rejection (RAR) algorithm). Given a target density $\psi:D\to{\mathbb{R}}_{+}$  and a proposal density $H: D\to \mathbb{R}_+$. Assume that there exists a constant $L<\infty$ such that  $\psi(\boldsymbol{x}) < LH(\boldsymbol{x})$ for all $\boldsymbol{x}$ in the domain $D$.
We introduce another random variable $u$ having uniform distribution in the unit interval, i.e. $u\thicksim U([0,1])$. Then the acceptance-rejection algorithm is given by
\begin{enumerate}
  \item Draw $X\thicksim H$ and $u\thicksim U([0,1])$.
  \item Accept $Y=X$ as a sample of $\psi$ if $u\leq \frac{\psi(\boldsymbol{X})}{LH(\boldsymbol{X})}$,
otherwise go back to step 1.
\end{enumerate}
\end{algorithm}
See also \cite{BHL13,Devroye,HLD04} for a discussion of related algorithms. For a discussion on how to select proposal densities see for instance \cite{BHL13} and the references therein.
 The  acceptance-rejection sampler works to sample from an unknown density  based on a proposal density.
 %This algorithm actually generates a realization of a time homogeneous Markov chain.

Acceptance-rejection sampling and importance sampling \cite[Section~3]{RobertCasella2004} are quite similar ideas. Both of them distort a sample from a  distribution in order to sample from another one.
However, there is a difference in the  selection of the constant $L>\psi(\boldsymbol{x})/H(\boldsymbol{x})$ for  $\boldsymbol{x}$ in the domain $D$.
The acceptance-rejection method does not work when $\sup_{\boldsymbol{x}\in D} \psi(\boldsymbol{x})/H(\boldsymbol{x})=\infty$, while importance
sampling is still available  \cite{OwenbookDraft}. In this paper, we only use the acceptance-rejection sampler to get samples of a given target density,
since we are interested in obtaining discrepancy bounds for MCQMC. More information on general strategies for generating nonuniform random variables can be found in the monographs \cite{Devroye, HLD04}.

 \begin{figure}\label{SobolVSPseudorandom}
 \includegraphics[trim=1.5cm 1cm 1.5cm 1cm, clip=true, totalheight=0.42\textheight, angle=360]{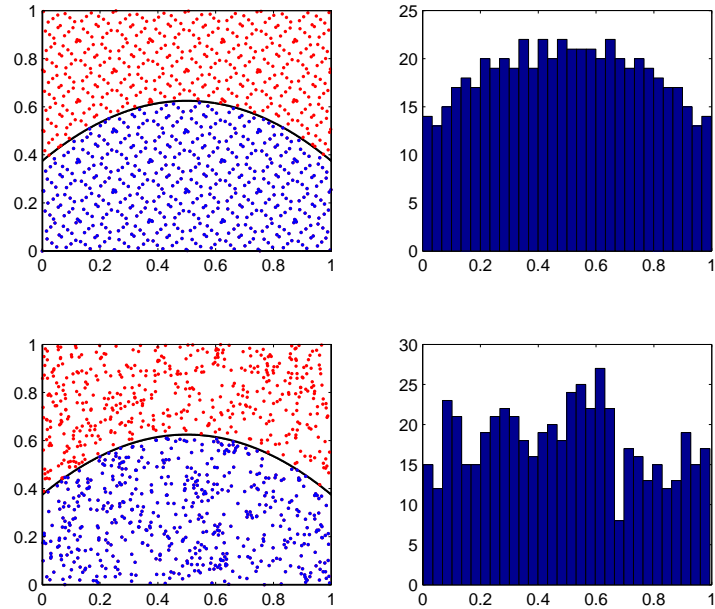}\\
  \caption{Different driver sequences: deterministic  and pseudo-random points (the total number of points in each case is $2^9$) and histograms are w.r.t. samples we accepted.}
\end{figure}

The acceptance-rejection algorithm with deterministic driver sequence is one special class of MCQMC.
From the superior distribution properties in terms of the discrepancy of the Sobol sequence  \cite{Sobol1967} one could expect an improvement in the discrepancy  of the samples obtained from the acceptance-rejection algorithm based on the Sobol sequence.  In one dimension the discrepancy we study is the Kolmogorov-Smirnov test   between the target distribution and the empirical distribution of the sample points.  For a given point set in the $s$-dimensional unit cube, the star discrepancy
  measures the difference between the  proportion of points in a subinterval of $[0,1]^s$ and the Lebesgue  measure of this subinterval. We defer the precise definition
  of discrepancy to Section~\ref{SecBackground}.

In this paper, we replace the  IID initial samples with an explicit construction of the driver sequence by using $(t,m,s)$-nets \cite{Josef,Niederriter1992} (obtained from the Sobol sequence).
Figure \ref{SobolVSPseudorandom} shows a  comparison between different driver sequences: deterministic points (Sobol points) and pseudo random uniform points (they both have $2^9$ points). The acceptance-rejection sampler works by only accepting  those points under the target density curve.  The difference of driver sequences  will affect the samples we obtain by the acceptance-rejection algorithm, hence the distribution properties of the points which were  accepted will be influenced. The right two figures in Fig.\ref{SobolVSPseudorandom} show the histograms of the points which we accepted in both cases. Note that the deterministic  samples  better estimate the density function.
 Our interest in this paper is in entirely deterministic methods. However, one could also use randomized quasi-Monte Carlo point sets \cite{Owen1, Owen2, Owen3} and  study a randomized setting.

\subsection{Previous work on deterministic acceptance-rejection algorithm }
 The deterministic acceptance-rejection  algorithm has also been discussed by Moskowitz and Caflisch~\cite{MC96} and Wang~\cite{Wang1999, Wang00}. Therein a smoothing technique was introduced to improve the numerical performance of the acceptance-rejection algorithm. Wang~\cite{Wang00} gave  a heuristic argument to indicate a convergence rate of order $N^{-\frac{s+2}{2(s+1)}}$. This argument assumes that the points in elementary intervals are uniformly distributed. Thus this reasoning is not fully deterministic. { Our lower bound on discrepancy (Theorem~\ref{lowerPsiN}) indicates that this reasoning does not apply  in our case.}
 The numerical experiments in \cite{Wang00} also indicate an improvement using a well chosen deterministic driver sequence
 (in this case the so-called Halton sequence \cite{Halton}) compared to a random driver sequence. Recently, Nguyen and \"{O}kten in \cite{NO2013} presented a consistency result of an acceptance-rejection algorithm for low-discrepancy sequences. This algorithm yielded good numerical performances on standard deviation and efficiency. However, proving an explicit convergence rate of the discrepancy for this algorithm is still  an open problem. See also \cite{MC95,NO2013} for numerical experiments
 using quasi-Monte Carlo point sets for the related problem of integrating indicator functions.

  It is worth noticing
 that all results given in previous work are empirical evidence and the discrepancy of samples is not directly investigated.
 Our work focuses on discrepancy properties of points produced by totally deterministic acceptance-rejection methods.  We also
 prove discrepancy bounds on deterministic acceptance-rejection algorithms, including an upper bound and a lower bound.  {The combination with the reduced acceptance-rejection sampler provides further evidence of the good
 performance of the deterministic method.} Our algorithm here may also be combined with similar algorithms like the acceptance-complement method, see for instance \cite[Section~II.5]{Devroye}.

Before  presenting the theoretical background, we briefly describe deterministic algorithms and some numerical results which show a convergence rate comparison  using Monte Carlo and quasi-Monte Carlo methods.
\section{Our results}
\subsection{Construction of driver sequence}\label{SubSecConstruction}
In this paper, we use low discrepancy point  sets given by $(t,s)$-sequences (see Definition~\ref{DeftmsNet} and Definition~\ref{DeftsSequence} below) in base $b$ as driver sequences. The first $b^m$ points of a $(t,s)$-sequence are a so called $(t,m,s)$-net in base $b$.
Explicit constructions of $(t,s)$-sequences in base $2$ have been found by Sobol~\cite{Sobol1967}, in prime base $b\ge s$ by Faure~\cite{Faure1982} and in prime-power base $b$ by Niederreiter~\cite{Niederriter1988}.  In all these  constructions $t$ depends only on $s$ but not on $m$. In practice, since digital nets (based on Sobol points) are included in the statistics toolbox of Matlab, this method is very easy to implement. People seeking  more discussion of construction methods can  also consult \cite[Chapters 4\&8]{Josef}.

In the following we describe the algorithm  and present some numerical results.
Since in general it is computationally too expensive to compute the supremum in the definition of the star-discrepancy exactly,  we use a so-called $\delta$-cover to estimate this supremum. An introduction to $\delta$-covers is provided in the appendix.
In the numerical discussion, the driver sequence is  generated by a $(t,m,s)$-net in base $2$. Specifically, we always use a Sobol sequence \cite{Sobol1967} to generate $(t,m,s)$-nets for our experiments.
\subsection{Deterministic algorithm for target densities defined on $[0,1]^s$}\label{SubsecCube}
 We consider now the case where the target density is defined on $[0,1]^{s-1}$. The following algorithm is a deterministic version of Algorithm~\ref{Basic_AR}.  For the proofs later, we need the technical assumption that the target density is pseudo-convex. The definition of pseudo-convexity is discussed in Section~\ref{SecBackground}.
\begin{algorithm}\label{algorithmAR}(Deterministic acceptance-rejection (DAR) algorithm in $[0,1]^s$).  Let the target density $\psi:[0,1]^{s-1}\to \mathbb{R}_+$,  where $s\geq 2$, be pseudo-convex. Assume that there exists a constant $L<\infty$ such that $\psi(\boldsymbol{x}) \leq L$ for  all $\boldsymbol{x}\in[0,1]^{s-1}$. Let $A=\{\boldsymbol{x}\in[0,1]^s:\psi(x_1,\ldots,x_{s-1})\geq L x_s\}$. Suppose we aim to obtain approximately $N$ samples from $\psi$.
\begin{itemize}
  \item [i)~] Let $M=b^m\ge\left \lceil N/(\int_{[0,1]^{s-1}}\psi(\boldsymbol{x})/L d\boldsymbol{x})\right\rceil$, where $m\in\mathbb{N}$ is the smallest integer satisfying this inequality.  Generate a $(t,m,s)$-net $Q_{m,s}=\{\boldsymbol{x}_0,\boldsymbol{x}_1,\ldots,\boldsymbol{x}_{b^m-1}\}$ in base $b$.
  \item [ii)~]Use the acceptance-rejection method for the points $Q_{m,s}$ with respect to the density $\psi$, i.e. we accept the point $\boldsymbol{x}_n$ if $\boldsymbol{x}_n \in A$, otherwise reject. Let $P_N^{(s)}=A\cap Q_{m,s}=\{\boldsymbol{z}_0,\ldots, \boldsymbol{z}_{N-1}\}$ be the sample set we accept.
  \item [iii)~]  Project the points  $P_N^{(s)}$ onto  the first $(s-1)$ coordinates.
  Let $P_N^{(s-1)}=\{\boldsymbol{y}_0,\ldots, \boldsymbol{y}_{N-1}\}\subseteq [0,1]^{s-1}$ be the projections of the points  $P_N^{(s)}$.
 \item [iv)~] Return the point set $P_N^{(s-1)}$.
\end{itemize}
\end{algorithm}

The following example shows a better convergence rate when using a low-discrepancy  driver sequence  rather than random point set. In each example the reported discrepancy for  the AR algorithm using a random diver sequence is the average of $10$ independent runs, which is throughout all the numerical experiments.

\begin{example}\label{Example24Dim}
In this example we consider a non-product target density in $[0,1]^4$.
Let target density  $\psi$ be
\begin{equation*}
    \psi(x_1,x_2,x_3,x_4)=\frac{1}{4}(e^{-x_1}+e^{-x_2}+e^{-x_3}+e^{-x_4}), ~~ (x_1,x_2,x_3,x_4) \in[0,1]^4.
\end{equation*}
\end{example}
%We choose $H(x_1,x_2,x_3,x_4)=1$ as the proposal density.  Algorithm~\ref{algorithmAR} is used to sample from density function in Example~\ref{Example24Dim}.  %We state this algorithm in more general way.
Figure~\ref{fig:Example24Dim} shows the discrepancy by using deterministic points and pseudo-random points as driver sequence.
For the RAR algorithm, we observe  a convergence rate of $N^{-0.482}$, whereas the DAR algorithm shows a convergence rate of the discrepancy of order $N^{-0.659}$.

 \begin{figure}[!htp]\begin{center}
   %Requires \usepackage{graphicx}
  \includegraphics[trim=0.5cm 0.25cm 0.5cm 0.5cm, clip=true, totalheight=0.4\textheight, angle=360]{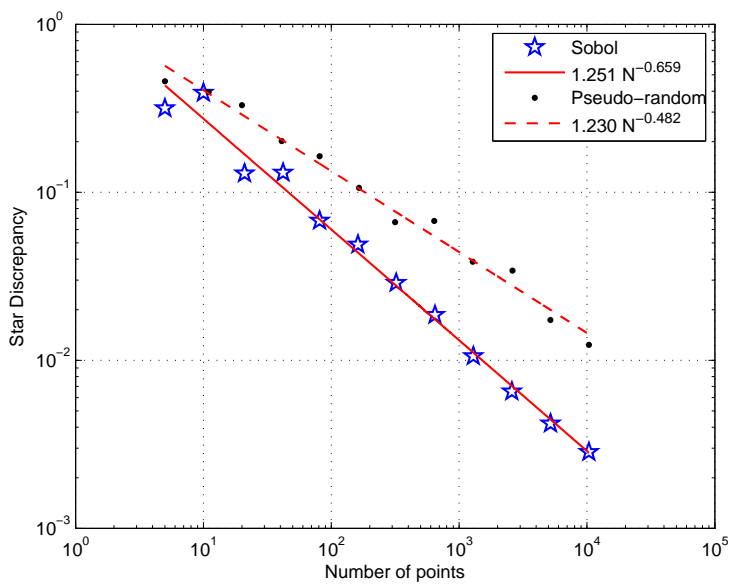}\\
  \caption{Convergence order of the star discrepancy of Example~\ref{Example24Dim}.}\label{fig:Example24Dim} \end{center}
\end{figure}
\subsection{Deterministic algorithm for target density defined in real state space}\label{SubsecR}
Now we extend the domain of the target density $\psi$ to  $\mathbb{R}^{s-1}$ with $s\ge 2$.
 Assume that there is  a proposal density function ${H}: \mathbb{R}^{s-1}\to \mathbb{R}_+$ such  that there exists a  constant ${L}<\infty$ such that $\psi(\boldsymbol{z})\le {L} {H}(\boldsymbol{z})$ holds for all $\boldsymbol{z}\in \mathbb{R}^{s-1}$.

 The inverse  Rosenblatt transformation  is used to generate samples from the proposal density
in the real state space $\mathbb{R}^{s-1}$.
 Let ${F}$ be the joint CDF of ${H}$ and ${F}_j(z_j|z_1,\ldots,z_{j-1})$ be the conditional CDF of the proposal density for $j = 1,\ldots,s-1$.
The  transformation ${T}$ is used to
generate points in $\mathbb{R}^{s-1}\times \mathbb{R}_+$ from the unit cube $[0,1]^{s}$, such that the projection of points onto the first $s-1$ coordinates has distribution $H$. More precisely,
let $T: [0,1]^{s}\to \mathbb{R}^s$ be the transformation  given by
\begin{equation}\label{InverseRosenblattTrans}\begin{array}{ll}
&\boldsymbol{z}=T(\boldsymbol{u})\\[0.25cm]
&=\left\{\begin{array}{lll}
  &z_1=F_1^{-1}(u_1), \\[0.25cm]
  &z_j=F_j^{-1}(u_j|u_1,\ldots,u_{j-1}), ~2\leq j\leq s-1,\\[0.25cm]
  &z_{s}= u_{s}H(z_1,\ldots,z_{s-1}).\\[0.25cm]
  \end{array}\right.\end{array}
\end{equation}
The first $s-1$ coordinates  are produced by  the inverse  Rosenblatt transformation which converts the points from the unit cube $[0,1]^{s-1}$ into $\mathbb{R}^{s-1}$.  The $s$th coordinate is uniformly distributed on the line
 $$\{(1-v)(z_{1},\ldots z_{s-1},0)+v(z_{1},\ldots z_{(s-1)},H(z_1,\ldots,z_{s-1}), 0\le v\le1\}$$ if
 $u_s$ is uniformly distributed in $[0,1]$. More details with respect to the Rosenblatt transformation and extensions can be found in \cite{Chentsov1967,Nohu2007,Rosenblatt1952}.

\begin{algorithm}\label{algorithmRd} (Deterministic acceptance-rejection algorithm in $\mathbb{R}^s$). Let an unnormalized target density function $\psi:\mathbb{R}^{s-1}\to \mathbb{R}_+$, where $s\ge 2$, be given. Let ${H}$ be a  proposal density $H:\mathbb{R}^{s-1}\to \mathbb{R}_+$, such that there exists a constant $L<\infty$ such that
$\psi(\boldsymbol{z})\leq L H(\boldsymbol{z})$ for all $\boldsymbol{z}\in \mathbb{R}^{s-1}$.
Let $A=\{\boldsymbol{z}\in\mathbb{R}^s:\psi(z_{1},\ldots,z_{s-1})\ge L H(z_{1},\ldots,z_{s-1}) z_s\}$. {Suppose we aim to obtain approximately $N$ samples from $\psi$.}
\begin{itemize}
                   \item [i)~]  Let $M= b^m\ge\left\lceil N\int_{[0,1]^{s-1}}H(\boldsymbol{x})d\boldsymbol{x}/(\int_{[0,1]^{s-1}}\psi(\boldsymbol{x})/L d\boldsymbol{x})\right\rceil$, where $m\in \mathbb{N}$ is the smallest integer satisfying this inequality. Generate  a $(t,m,s)$-net $Q_{m,s}=\{\boldsymbol{x}_{0},\boldsymbol{x}_{1},\ldots,\boldsymbol{x}_{M-1}\}$ in base $b$.
                    \item [ii)~]  Transform the  points into $\mathbb{R}^{s-1}\times \mathbb{R}_+$ from $[0,1]^s$ using the  transformation $T$ given in \eqref{InverseRosenblattTrans} to obtain $\{T(\boldsymbol{x}_0),\ldots,T(\boldsymbol{x}_{M-1})\}$.
                    \item [iii)~] Take the acceptance-rejection method for the sample $T(\boldsymbol{x}_{n})$ with respect to $H$ and $\psi$ in $\mathbb{R}^{s-1}\times\mathbb{R}_+$, i.e. accept the point $T(\boldsymbol{x}_{n})$ if $T(\boldsymbol{x}_{n})\in A$, otherwise reject.
                    Let $P_N^{(s)}=A\cap T(Q_{m,s})=\{\boldsymbol{z}_0,\ldots, \boldsymbol{z}_{N-1} \}$.
                    \item [iv)~]  Project the points $P_N^{(s)}$ we accepted onto  the first $(s-1)$-dimensional space. Denote the first $s-1$ coordinates of the points we accept by the acceptance-rejection method by $P_N^{(s-1)}=\{\boldsymbol{y}_0,\ldots, \boldsymbol{y}_{N-1}\}\subseteq \mathbb{R}^{s-1}$.
                     \item [v)~] Return the point set $P_N^{(s-1)}$.
\end{itemize}
\end{algorithm}
 We provide an example to demonstrate the performance of Algorithm~\ref{algorithmRd}.
\begin{example}\label{Example3} Let the target density function  be given by
\begin{equation*}
  \psi(x_1,x_2)=\left\{\begin{array}{ll}
\displaystyle\frac{4}{\pi}e^{-(x_1+x_2)}(x_1x_2)^{1/2},& \quad x_1,x_2>0,\\
    0,& \quad \mbox{ else.}
    \end{array}\right.
\end{equation*}
\end{example}
The proposal density function  $H$,  which we  use to do the  acceptance-rejection to generate samples of $\psi(x_1,x_2)$,  is chosen as
 \begin{equation*}
   H(x_1,x_2)=\left\{ \begin{array}{ll}
  \displaystyle\frac{1}{4}, &0\leq x_1,x_2\leq 1,\\[0.2cm]
  \displaystyle\frac{1}{4x_2^2}, &0\leq x_1\leq 1,x_2 >1,\\[0.2cm]
    \displaystyle\frac{1}{4x_1^2}, &x_1 >1,0\leq x_2\leq 1,\\ [0.2cm]
  \displaystyle\frac{1}{4x_1^2x_2^2}, &x_1,x_2 >1,\\[0.2cm]
  0, &\quad \mbox{else.}\\
      \end{array}\right.
 \end{equation*}
For this choice of $H$, we use transform $T$ defined in Equation~\eqref{InverseRosenblattTrans} to obtain samples from $H$.
The sample $(x_{j,1},x_{j,2})$ is given by the following transformation
\begin{equation*}\begin{array}{ll}
x_{j,1}=&\left\{\begin{array}{lll}
  2u_{j,1}, &0\le u_{j,1}\le 1/2,\\
   1/2(1-u_{j,1}),& 1/2<u_{j,1}\le 1,\\
    \end{array}\right.\end{array}
\end{equation*}
\begin{equation*}\begin{array}{ll}
x_{j,2}=&\left\{\begin{array}{lll}
  2u_{j,2}, &0\le u_{j,2}\le 1/2,\\
   1/2(1-u_{j,2}),& 1/2<u_{j,2}\le 1.\\
    \end{array}\right.\end{array}
\end{equation*}
 Note that $(u_{j,1},u_{j,2})$ is the driver sequence given by a $(t,m,2)$-net in base $b$.

The order of the star discrepancy is demonstrated in  Figure~\ref{OrderGamma3-2} where $N$ is the number of {accepted} samples. The numerical experiments  show that the star discrepancy converges  at a rate of $N^{-0.720}$ for this example using quasi-Monte Carlo samples as proposal.  The RAR algorithm converges with order $N^{-0.390}$. Again, the DAR sampler outperforms the RAR sampler.

\begin{figure}[!htp]\begin{center}
  %Requires \usepackage{graphicx}
 \includegraphics[trim=0.5cm 0.25cm 0.5cm 0.5cm, clip=true, totalheight=0.4\textheight, angle=360]{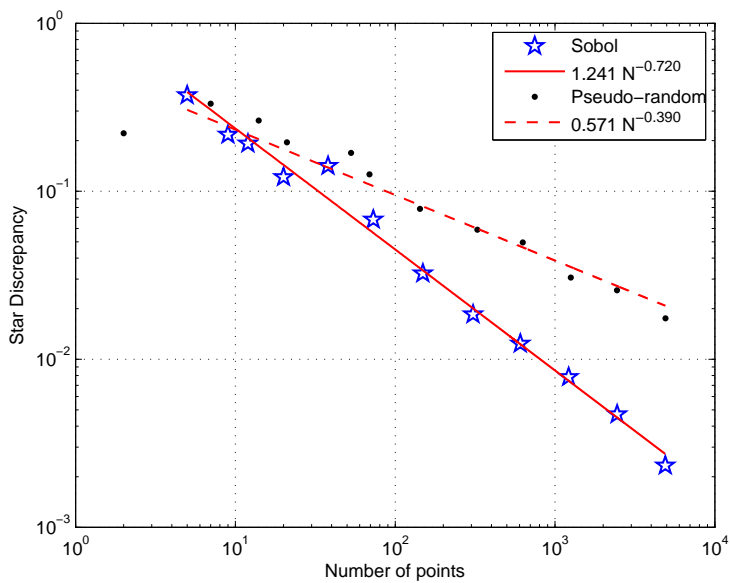}\\
 \caption{ Convergence order of the star discrepancy of Example~\ref{Example3}.}\label{OrderGamma3-2}\end{center}
\end{figure}

\subsection{A deterministic reduced acceptance-rejection sampler}\label{SubsecSum}
In this subsection we consider  an extension of the DAR sampler.  The random version of this reduced method was recently introduced by Barekat and Caflisch in \cite{BC2013}.  For a target density function $\psi$, we carefully select $H$ such that for $\psi-H$ and $H$ the inverse CDF can be computed. For the case $\psi(\boldsymbol{x})> H(\boldsymbol{x})$, we write $\psi=(\psi-H)+H$ and get samples according to $\psi-H$ and $H$  respectively.

 Figure~\ref{FishProblem} illustrates this method. The sample sets of $\psi$ can be divided into three subsets, $R_{1,1}, R_{1,2}$ and $R_{2,2}$, where $R_{1,2}$ and $R_{2,2}$ can be directly generated by using the inverse CDF of $H$ and $\psi-H$ in a certain range. The acceptance-rejection method is only used to obtain $R_{1,1}$.  Compared with the ordinary acceptance-rejection sampler, one obvious merit of this method is that we do not require $\psi(\boldsymbol{x})\le H(\boldsymbol{x})$ in the whole domain. Also, this method might give better convergence rates {since $R_{1,2}$ and $R_{2,2}$ are obtained via inversion and therefore have low discrepancy}.  Algorithm~\ref{Alg-RAR} gives a simple version of the improved method.  More discussion of a general version is available in Section~\ref{SecImprovedAlg}.

 \begin{figure}[!htp]\begin{center}
  %Requires \usepackage{graphicx}
 \includegraphics[trim=0.5cm 0.25cm 0.5cm 0.5cm, clip=true, totalheight=0.4\textheight, angle=360]{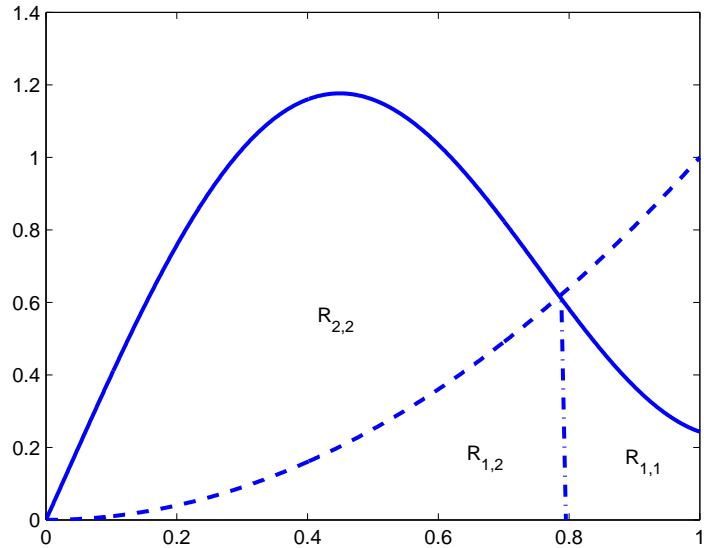}\\
 \caption{The solid line represents the target density $\psi$ and the dashed line is the proposal density $H$. }\label{FishProblem}\end{center}
\end{figure}

\begin{algorithm}\label{Alg-RAR}{(Deterministic reduced acceptance-rejection (DRAR) algorithm)} Let $\psi:[0,1]\to \mathbb{R}_+$ be a target density.  Choose a proposal density $H$ such that $\psi-H$ and $H$  can be sampled directly.
Let\begin{align*}
    &\mathcal{S}:=\{x\in [0,1]: \psi(x)< H(x)\}\nonumber \\
    \mbox{ and }&\\
    &\mathcal{L}:=\{x\in [0,1]: \psi(x)\ge H(x)\}.\nonumber \\
\end{align*} Assume that $\int_{\mathcal{S}}\psi(x)dx/\int_{[0,1]}\psi(x)dx,\int_{\mathcal{L}}H(x)dx/\int_{[0,1]}\psi(x)dx$ and $\int_{\mathcal{L}}(\psi-H)(x)dx/\int_{[0,1]}\psi(x)dx$ can be calculated or estimated.
Let $F_{H,\mathcal{S}}^{-1},F_{H,\mathcal{L}}^{-1}$ be the inverse  CDF of the proposal density $H$ in the domain $\mathcal{S}$ and $\mathcal{L}$ respectively and $F_{\psi-H,\mathcal{L}}^{-1}$ be the inverse CDF with respect to $\psi-H$ in $\mathcal{L}$.
Suppose we aim to generate approximately $N$ samples from $\psi$. Let
$$N_1=\left\lceil N\frac{ \int_{\mathcal{S}}\psi(x)dx}{\int_{[0,1]}\psi(x)dx}\right\rceil,  N_2=\left\lceil N\frac{\int_{\mathcal{L}}H(x)dx}{\int_{[0,1]}\psi(x)dx}\right\rceil
\mbox{ and } N_3=\left\lceil N\frac{\int_{\mathcal{L}}(\psi-H)(x)dx}{\int_{[0,1]}\psi(x)dx}\right\rceil.$$
{\begin{itemize}
 \item [i)~] Let $\{\boldsymbol{x}_0,\boldsymbol{x}_1, \boldsymbol{x}_2,\ldots\}\subset [0,1]^2$ be a $(t,2)$-sequence in base $b$.

 \item [ii)~]   Use the acceptance-rejection method with the target density $\psi$ and the proposal density $H$ on the domain $\mathcal{S}$ using  $\{\boldsymbol{x}_0,\boldsymbol{x}_1, \ldots,\boldsymbol{x}_{M-1}\}$ as driver sequence. Choose $M$ such that $N_1$ points are accepted by the DAR algorithm. Compute $F^{-1}_{H,\mathcal{S}}(\boldsymbol{x}_n)$ for $n=0,1,2,\ldots, M-1$. Let $\boldsymbol{z}_0,\boldsymbol{z}_1, \ldots,\boldsymbol{z}_{N_1-1}$ be the accepted points. Label the point set as $R_{1,1}$.
 \item [iii)~]  Compute the points $F_{H,\mathcal{L}}^{-1}(\boldsymbol{x}_n)$ for $n=0,1,\ldots, N_2-1$. Let $R_{1,2}=\{F_{H,\mathcal{L}}^{-1}(\boldsymbol{x}_n):0\le n< N_2\}$.
  \item [iv)~]  Compute the points $F_{\psi-H,\mathcal{L}}^{-1}(\boldsymbol{x}_n)$ for $n=0,1,\ldots, N_3-1$. Let $R_{2,2}=\{F_{\psi-H,\mathcal{L}}^{-1}(\boldsymbol{x}_n):0\le n< N_3\}$.
 \item [v)~]  Project the points in $R_N=R_{1,1}\cup R_{1,2}\cup R_{2,2}$ onto  the first  coordinate. Return the point set $R_N^{(1)}$.
    \end{itemize}}
\end{algorithm}
 Since the inverse transform is a measure-preserving transformation, it can preserve the uniformities of the driver sequence. Thus $R_{1,2}$ and $R_{2,2}$ are low discrepancy point sets. The following example verifies the efficiency  of the DRAR algorithm. A theoretical result about the discrepancy properties of samples obtained by this class of algorithms is provided in  Theorem~\ref{triangleDisc}.

\begin{example}\label{ExampleSumCase} Let $\psi(x)=\sin(4x)+x^2$  be a density function defined on $[0,1]$. Instead of seeking  a proposal density $H$ such that  $\psi(x)\le H(x)$, we notice that inversion  can be implied to $\sin(4x)$ and $x^2$ independently. However, it can not work for their  sum. Choose $H(x)=x^2$. We only do deterministic acceptance-rejection with respect to the target density $\psi$ and proposal density $H$ in the subinterval $\mathcal{S}=(\pi/4,1]$.  In the remaining range $\mathcal{L}=[0,\pi/4]$,  we apply the inverse transformation on $H$ and $\psi-H$ to obtain samples based on a deterministic driver sequence.\end{example}
The discrepancy of the point set generated by Algorithm~\ref{Alg-RAR}  converges at the rate of  $N^{-0.929}$, which is significantly better than the $N^{-0.501}$ convergence rate of a random driver sequence, see Figure~\ref{figDiscrepSumcase}.

\begin{figure}
\includegraphics[trim=0.5cm 0.25cm 0.5cm 0.5cm, clip=true, totalheight=0.4\textheight, angle=360]{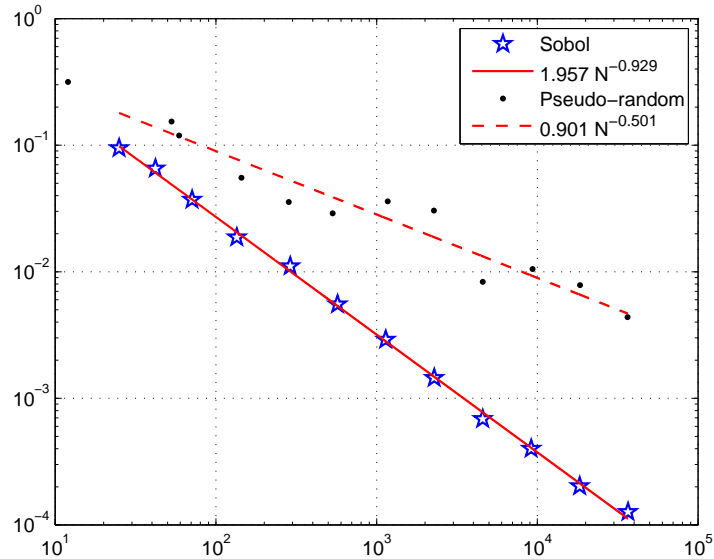}\\
   \caption{Convergence order of the star discrepancy of Example~\ref{ExampleSumCase}.}\label{figDiscrepSumcase}
\end{figure}

\section{Background on  discrepancy theory and $(t,m,s)$-nets}\label{SecBackground}
In this section we  first establish some notation and  some useful definitions and then obtain theoretical results.
First we introduce the definition of $(t,m,s)$-nets in base $b$ (see  \cite{Josef}) which we use as the driver sequence throughout the paper. The following fundamental definitions of elementary interval and fair sets are used to define a $(t,m,s)$-net and $(t,s)$-sequence in base $b$.
\begin{definition} ({b-adic elementary interval}). Let $b\geq 2$ be an integer. An $s$-dimensional $b$-adic elementary interval is an interval of the form
\begin{equation*}
    \prod_{i=1}^s\left[\displaystyle\frac{a_i}{b^{d_i}},\displaystyle\frac{a_i+1}{b^{d_i}}\right)
\end{equation*}
with integers $0\leq a_i<b^{d_i}$ and $d_i\geq 0$ for all $1\leq i\leq s$. If $d_1,\ldots,d_s$ are such that $d_1+\cdots+d_s=k$, then we say that the elementary interval is of order $k$.% We denote the set of all $b$-adic elementary intervals of order $k$ by $\mathcal{E}_k$.
\end{definition}

 \begin{definition}({fair sets}).~For a given set $P_N=\{\boldsymbol{x}_0,\boldsymbol{x}_1,\ldots,\boldsymbol{x}_{N-1}\}$ consisting of $N$ points in $[0,1)^s$, we say for a subset $J$ of $[0,1)^s$ to be fair with respect to $P_N$, if
\begin{equation*}
   \displaystyle\frac{1}{N}\sum_{n=0}^{N-1}1_J(\boldsymbol{x}_n)=\lambda(J),
\end{equation*}
where $1_J(\boldsymbol{x}_n)$ is the indicator function of the set $J$ and $\lambda$ is the Lebesgue measure. \end{definition}

\begin{definition}\label{DeftmsNet} ({$(t,m,s)$-nets in base b}).~For a given dimension $s\geq 1$, an integer base $b\geq 2$, a positive integer $m$ and an integer $t$ with $0\leq t\leq m$,
a point set $Q_{m,s}$ of $b^m$ points in $[0,1)^s$ is called a  $(t,m,s)$-nets in base $b$ if the point set $Q_{m,s}$ is fair with respect to all b-adic s-dimensional elementary intervals of order at most $m-t$.\end{definition}

\begin{definition}\label{DeftsSequence}($(t,s)$-sequence).~For a given dimension $s\geq 1$, an integer base $b\geq 2$ and a positive integer $t$, a sequence $\{\boldsymbol{x}_0, \boldsymbol{x}_1,\ldots\}$ of points in $[0,1)^s$ is called a  $(t,s)$-sequence in base $b$ if for all integers $m\ge t$ and $k\ge 0$, the point set consisting of the points $\boldsymbol{x}_{kb^m},\ldots,\boldsymbol{x}_{kb^m+b^m-1}$ forms a $(t,m,s)$-net in base $b$.
\end{definition}
The concept of discrepancy is introduced in \cite{Niederriter1974} to measure the deviation of
 a sequence from the uniform distribution.  Now we give the definition of the so-called
 star discrepancy which enables us to distinguish  the quality of point sets with respect to the uniform distribution.
\begin{definition}\label{stardiscrepancy} ({star discrepancy}).~Let  $P_N=\{\boldsymbol{x}_0, \boldsymbol{x}_1,\ldots,\boldsymbol{x}_{N-1}\}$ be a point set in  $[0,1)^s$. The star discrepancy $D_N^*$ is defined by
\begin{equation*}
    D_N^*(P_N)=\sup_{J\subset [0,1)^s}\left| \displaystyle\frac{1}{N}\sum_{n=0}^{N-1}1_J(\boldsymbol{x}_n)-\lambda(J)\right|,
\end{equation*}
where the supremum is taken over all  $J=\prod_{i=1}^s [0,\beta_i)\subseteq [0,1)^s$.\end{definition}
See Figure~\ref{Fig:StarDiscrepancy} for an illustration of the concept of discrepancy in the unit square.
\begin{figure}[!htp]\begin{center}
  %Requires \usepackage{graphicx}
 \includegraphics[trim=1cm 0.25cm 0.5cm 0.5cm, clip=true, totalheight=0.4\textheight, angle=360]{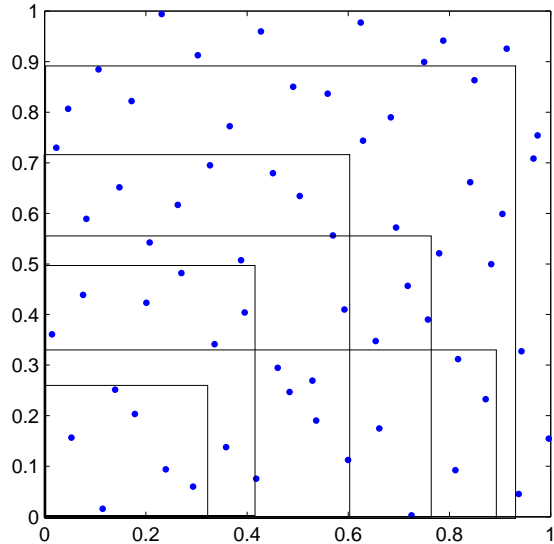}\\
 \caption{The discrepancy measures the difference between the proportion of points in each rectangle $J$ and the Lebesgue measure of $J$. The star discrepancy
  is defined by the supremum of discrepancies over all rectangles $J$. If we change $J$ to be  convex sets, we obtain the so-called isotropic discrepancy.}\label{Fig:StarDiscrepancy}
 \end{center}
\end{figure}

If we extend the supremum in Definition \ref{stardiscrepancy} over all convex sets in $[0,1]^s$, we get another interesting discrepancy, the so-called isotropic discrepancy.
 It is  another measure of the distribution properties of point sets with respect to convex sets.
\begin{definition}  ({isotropic discrepancy}). \label{isotropic_dis} ~Let  $P_N=\{\boldsymbol{x}_0,\boldsymbol{x}_1,\ldots,\boldsymbol{x}_{N-1}\}$  be a point set in  $[0,1)^s$. The isotropic discrepancy $J_N$ is defined to be
\begin{equation*}
    J_N(P_N)=\sup_{J\subset\mathcal{C}}\left| \displaystyle\frac{1}{N}\sum_{n=0}^{N-1}1_J(\boldsymbol{x}_n)-\lambda(J)\right|,
\end{equation*}
where $\mathcal{C}$ is the family of all convex subsets of $[0,1)^s$.\end{definition}

For  further reading about the definition and properties of discrepancy, we refer for instance to \cite{Josef,Niederriter1974}.

For our purposes here we need the definition of pseudo-convex sets which we introduce in the following (see also \cite[Definition 2]{CJJ2011} and Figure~\ref{FigPseudoconvex} for an example).
 %First we give a definition of pseudo-convex sets which we consider here.  .
\begin{definition}\label{pseudoConvexSet}(pseudo-convex set). Let $A$ be an open subset of $[0,1]^s$ such that there exists a collection of $p$ convex  subsets $A_1,\ldots,A_p$ of $[0,1]^s$
satisfying\begin{enumerate}
            \item $A_i\cap A_j=\emptyset$ for $i\neq j$;
            \item $A\subseteq ( A_1\cup\cdots \cup A_p)$ of $[0,1]^s$ ;
            \item either $A_j$ is a convex part of $A$ ($A_j\subseteq A$ for $j=1,\ldots,q$) or the complement of $A$ with respect to $A_j$, $A_j'=A_j\backslash A$ is convex.
          \end{enumerate}
Then $A$ is called a pseudo-convex set and $A_1,\ldots, A_p$ is an admissible convex covering for $A$ with $p$ parts and with $q$ convex parts of $A$.\end{definition}

\begin{figure}\label{FigPseudoconvex}
\includegraphics[trim=0.5cm 0.25cm 0.5cm 0.5cm, clip=true, totalheight=0.4\textheight, angle=360]{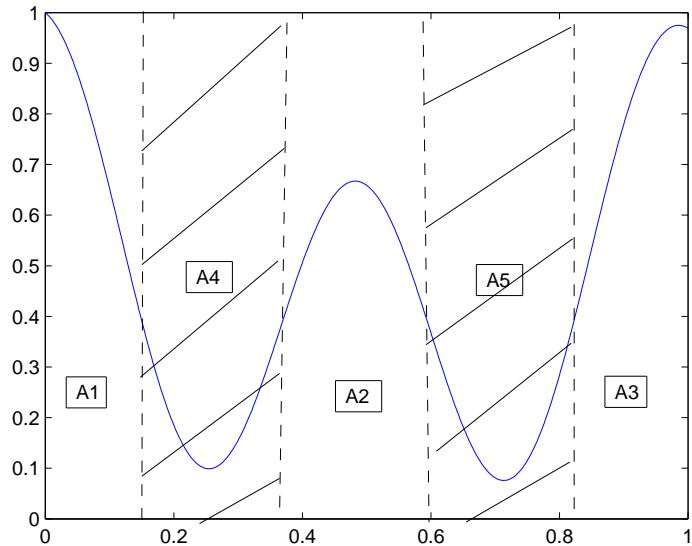}\\
   \caption{Shows a  pseudo-convex set in the unit square given by the area below graph of the density function and its admissible convex covering.
   %We use a set in two dimension to illustrate the concept of pseudo-convex set and its admissible convex covering.
  Let A be given by the graph under the curve. Then $A_i, i=1,\ldots,5$, is an admissible convex covering of $A$, where $A_1, A_2$ and $A_3$ are convex parts in $A$ but $A_4$ and $A_5$ are rectangles shadowed region covering the remaining part of $A$. The regions $A_4\backslash A$ and $A_5\backslash A$ are convex.}
\end{figure}

\begin{remark}\label{remark-Pseudoconvex}For convenience, we call a nonnegative function pseudo-convex if and only if  the region below its graph is a pseudo-convex  set.\end{remark}

 Next we present a bound on the isotropic discrepancy of points generated by  $(t,m,s)$-nets. A detailed proof is given in Appendix~\ref{AppendixUpperBound}.

\begin{lemma}\label{IsoDisBound} Let the point set $Q_{m,s}=\{\boldsymbol{x}_0,\boldsymbol{x}_1,\ldots, \boldsymbol{x}_{M-1}\}\subseteq[0,1]^s$ be a  $(t,m,s)$-net in base $b$ where $M=b^m$.
 For the  isotropic discrepancy of $Q_{m,s}$  we have
\begin{equation*}
J_M(Q_{m,s})\leq 2sb^{t/s}M^{-1/s}.
\end{equation*}
\end{lemma}

A slightly weaker result than Lemma~\ref{IsoDisBound} can also be obtained from \cite[Korollar~3]{NW75}.
\begin{lemma}\label{StarIsotropic} For any point set $P_N$ in $[0,1]^s$ we have
\begin{equation*}
J_N(P_N) \le 2 s \left(\frac{4s}{s-1}\right)^{(s-1)/s} ( D_N^*(P_N) )^{1/s}.
\end{equation*}
\end{lemma}
Further it is known from \cite{Kr06} that the star discrepancy of a $(t,m,s)$-net $Q_{m,s}$ in base $b$, where $M = b^m$, satisfies
\begin{equation*}
D_{M}^*(Q_{m,s}) \le M^{-1} b^{t} (\log M)^{s-1} \frac{b^s}{(b+1) 2^s (s-1)! (\log b)^{s-1}} + C_s M^{-1} b^{t} (\log M)^{s-2},
\end{equation*}
for some constant $C_s > 0$. These two inequalities therefore yield a convergence rate of order $M^{-1/s} (\log M)^{1-1/s}$.

The following lemma will be used to get a discrepancy bound  for a point set on a pseudo-convex set. It is an extension of \cite[Lemma 5]{CJJ2011} to the unit cube.
\begin{lemma} \label{lemmaConvex}
Let $A$ be a pseudo-convex subset of $[0,1]^s$ with admissible convex covering of $p$ parts with $q$ convex parts of $A$. Then for any point set $P_N=\{\boldsymbol{x}_0,\boldsymbol{x}_1,\ldots,\boldsymbol{x}_{N-1}\}\subseteq [0,1]^s$ we have
\begin{equation*}
    \left|\displaystyle\frac{1}{N}\sum_{n=0}^{N-1}1_A(\boldsymbol{x}_n)-\lambda(A)\right|\leq (2p-q) J_N(P_N).
\end{equation*}
\end{lemma}

\section{Discrepancy investigation of deterministic acceptance-rejection sampler}\label{TheoricalResults}

The first result we get is a discrepancy bound with respect to the target density of samples generated by the acceptance-rejection algorithm with deterministic driver sequences.
The star discrepancy of points generated by the acceptance-rejection algorithm with respect to the target density  converges  at the rate of $N^{-1/s}$, where $N$ is the number of accepted samples. See Theorem~\ref{Maintheorem} for details.
The proof  uses a bound on the discrepancy of our driver sequence with respect to convex sets (which is called isotropic discrepancy, see Definition~\ref{isotropic_dis} for details).

\subsection{Upper bound }\label{UpboundCube}
Let an unnormalized  density function $\psi:[0,1]^{s-1}\to \mathbb{R}_+$ be pseudo-convex, and $\int_{[0,1]^{s-1}}\psi (\boldsymbol{z})d\boldsymbol{z} >0$, but not necessarily $1$.
 Assume that there exists a constant $L<\infty$ such that $\psi(\boldsymbol{x}) \leq L$ for  all $\boldsymbol{x}\in[0,1]^{s-1}$. Let the subset under the graph of $\psi/L$  be defined as
\begin{equation}\label{SubsetA}
  A=\{\boldsymbol{x}\in[0,1]^s:\psi(x_1,\ldots,x_{s-1})\geq L x_s\},
\end{equation}
 which is pseudo-convex in $[0,1]^s$ as $\psi$ is a  pseudo-convex function. Assume that there is an admissible convex covering of $A$ with $p$ parts and with $q$ convex parts of $A$.
  Without loss of generality, let $A_1,\ldots,A_q$ be the convex subsets of $A$ and $A_{q+1},\ldots,A_p$, such that  $A_j'=A_j\backslash A$ is convex for  $q+1\leq j\leq p$.

 The definition of the star discrepancy of a point set $\{\boldsymbol{y}_0,\boldsymbol{y}_1,\ldots, \boldsymbol{y}_{N-1}\}$ with respect to a
 density function $\psi$ is given as follows.

\begin{definition} Let $\psi:[0,1]^{s-1}\to \mathbb{R}_+$ be an unnormalized  target density. Let $\{\boldsymbol{y}_0,\boldsymbol{y}_1,\ldots, \boldsymbol{y}_{N-1}\}$ be a
 point set in $[0,1]^{s-1}$. The star discrepancy of $\{\boldsymbol{y}_0,\boldsymbol{y}_1,\ldots, \boldsymbol{y}_{N-1}\}$ with respect to the  density $\psi$ is defined by
\begin{equation*}\label{DefDiscrepancyWRTdensity}
   D_{N,\psi}^*(P_N^{(s-1)})=\sup_{\boldsymbol{t}\in[0,1]^{s-1}} \left|\displaystyle\frac{1}{N}\sum_{n=0}^{N-1}1_{[\boldsymbol{0},\boldsymbol{t})}({\boldsymbol{y}_n})
    -\displaystyle\frac{1}{C}\int_{[\boldsymbol{0},\boldsymbol{t})}\psi({\boldsymbol{z}})d\boldsymbol{z}\right|,
\end{equation*}
where $C=\int_{[0,1]^{s-1}}\psi(\boldsymbol{z})d\boldsymbol{z}$ and $[\boldsymbol{0},\boldsymbol{t})=\prod_{j=1}^{s-1}[0,t_j)$.
\end{definition}
\begin{remark}
Note that $\frac{1}{C}\psi$ is a probability density function on $[0,1]^{s-1}$. Thus  the discrepancy  in Definition \ref{DefDiscrepancyWRTdensity}
measures the difference between the distribution $\frac{1}{C}\psi$ and the empirical distribution of the sample points with respect to the test
 sets $[\boldsymbol{0}, \boldsymbol{t})$ for $\boldsymbol{t}\in [0,1]^{s-1}$.
\end{remark}

\begin{theorem} \label{Maintheorem}
Let the unnormalized  density function $\psi:[0,1]^{s-1}\to \mathbb{R}_+$, with $s\geq 2$, be pseudo-convex.  Assume that there is an admissible convex covering of
 $A$ given by Equation~\eqref{SubsetA}
 with $p$ parts and with $q$ convex parts of $A$.
Then the  discrepancy of the point set $\{\boldsymbol{y}_0,\boldsymbol{y}_2,\ldots, \boldsymbol{y}_{N-1}\} \subseteq [0,1]^{s-1}$ generated by
Algorithm \ref{algorithmAR} using a $(t,s)$-sequence in base $b$, for large enough $N$, satisfies
\begin{equation*}
   D_{N,\psi}^*(P_N^{(s-1)})
    \le 8C^{-1}L s b^{t/s}(2p-q) N^{-1/s},
  \end{equation*}
  where $C=\int_{[0,1]^{s-1}}\psi({\boldsymbol{z}})d\boldsymbol{z}$ and $\psi (\boldsymbol{x})\le L$ for all $\boldsymbol{x}\in[0,1]^{s-1}$.
\end{theorem}

We postpone the proof of this theorem to Appendix~\ref{AppendixUpperBound}.

\subsection{Lower bound}\label{Lowerbound}
In this section, we provide a lower bound on the star discrepancy with respect to a  convex density function. The general idea is to find, for a given driver point set, a density function satisfying a certain convergence rate.
{
\begin{theorem}\label{lowerPsiN} Let $P_M$ be an arbitrary  point set   in $[0,1]^s$. Then there exists a concave density function $\psi$ defined in $[0,1]^{s-1}$ such that, for $N$ samples generated by the acceptance-rejection algorithm with respect to $P_M$ and $\psi$, we have
\begin{equation*}
 D_{N,\psi}^*(P_N)\ge c_sN^{-\frac{2}{s+1}},
\end{equation*}
where $c_s >0$ is independent of $N$ and $P_M$ but depends on $s$.
\end{theorem}}

A detailed proof is provided in Appendix~\ref{AppendixLowerbound}. We would like to point out that
the lower bound also limits the convergence rates which we can obtain in our current approach via convex sets.
Note that a concave function is also pseudo-convex as defined in Remark~\ref{remark-Pseudoconvex}.

Additionally,  note that \cite{Beck} (in dimension $s=2$) and \cite{Stu77} (for dimension $s > 2$)
showed the existence of points with discrepancy with respect to convex sets bounded from above by $N^{-2/(s+1)} (\log N)^{c(s)}$ (where $c(s)$ is a function of only $s$). This would yield an improvement of our results from $N^{-1/s}$ to $N^{-2/(s+1)}(\log N)^{c(s)}$, however, those constructions are not explicit and can therefore not be used in computation.

\subsection{Generalization to real state space}\label{BoundsGeneralization}
We consider now the case where the target density is defined on $\mathbb{R}^{s-1}$ with $s\ge 2$.  The aim is to show a discrepancy bound on samples generated by the deterministic acceptance-rejection method.
The discrepancy  with respect to a given density function $\psi:\mathbb{R}^{s-1}\to \mathbb{R}_+$ is defined as follows.
\begin{definition}
Let $P_N=\{\boldsymbol{z}_0,\boldsymbol{z}_1,\ldots,\boldsymbol{z}_{N-1}\}$ be a point set in $\mathbb{R}^{s-1}$. Let $\psi:\mathbb{R}^{s-1}\to \mathbb{R}_+ $ be an unnormalized probability density function. Then the star discrepancy
 $D_{N,\psi}^*(P_N)$ is defined by
\begin{equation*}
    D_{N,\psi}^*(P_N)=\sup_{\boldsymbol{t}\in \mathbb{R}^{s-1}}\left|\frac{1}{N}\sum_{n=0}^{N-1}1_{(-\boldsymbol{\infty},\boldsymbol{t}]}(\boldsymbol{z}_n)-\frac{1}{C}\int_{(-\boldsymbol{\infty},\boldsymbol{t}]} \psi(\boldsymbol{z})d\boldsymbol{z}\right|,
   \end{equation*}
where $C=\int_{\mathbb{R}^{s-1}} \psi(\boldsymbol{z})d\boldsymbol{z}$ and $(-\boldsymbol{\infty},\boldsymbol{t}]=\prod_{j=1}^{s-1}(-\infty, t_j]$ for $\boldsymbol{t}=(t_1, \ldots,t_{s-1})$.
\end{definition}

We use the transformation $T$ given in Equation~\eqref{InverseRosenblattTrans} to generate samples of ${H}$. For the sake of investigating discrepancy, the following result is helpful. The lemma shows that the transformations $T$ and its inversion $T^{-1}$  are both measure-preserving. For the proofs later, we assume that the proposal density $H$ is a product measure, i.e. $H=\prod_{j=1}^{s-1}H_j$, where $H_j$ is the marginal density with respect to $z_j$. In our numerical examples, the proposal density is not necessarily of product type.

\begin{lemma} The transformation $T$ from the $s$-dimensional unit cube to $\mathbb{R}^{s-1}\times \mathbb{R}_+$ given in (\ref{InverseRosenblattTrans})
 is measure-preserving, i.e. $\mbox{Volume}(T(D))= \mbox{Volume}(D)$ holds for any measurable set $D\subseteq [0,1]^s$. This is true for $T^{-1}$ as well.
\end{lemma}

To prove a bound on the discrepancy of the samples generated by Algorithm~\ref{algorithmRd},
the following assumption is needed.
\begin{assumption}\label{AssumptionGeneral} Let $\psi$ be the target density and ${H}$ be a product measure proposal density function, which is chosen such that its inverse CDF can be computed.
Let $A=\{\boldsymbol{z}\in\mathbb{R}^s:\psi(z_{1},\ldots,z_{s-1})\ge Lz_s H(z_{1},\ldots,z_{s-1})\}$ and the transformation $T^{-1}$  is defined as the inversion of transform $T$. Then we assume that $T^{-1}(A)$ is pseudo-convex.
\end{assumption}

 As the mappings $T$ and $T^{-1}$ are measure preserving, and since there are the same number of samples in an arbitrary subset $D\subseteq [0,1]^{s}$ and the corresponding subset $T(D)\subseteq \mathbb{R}^{s-1}\times \mathbb{R}_+$, we can consider the discrepancy in the unit cube instead of that in $\mathbb{R}^{s-1}\times \mathbb{R}_+$. Following by  similar proof arguments as for Theorem \ref{Maintheorem} and Theorem~\ref{lowerPsiN}, we obtain the same discrepancy bounds including an upper bound and a lower bound for the general density $\psi$ defined in the real state space $\mathbb{R}^{s-1}$.

\begin{theorem}\label{GneralMaintheorem}  Let the unnormalized target density  $\psi:\mathbb{R}^{s-1}\to \mathbb{R}_+$ and the proposal density $H:\mathbb{R}^{s-1}\to\mathbb{R}_+$ satisfy Assumption \ref{AssumptionGeneral}. Then the  discrepancy of the point set $P_N^{(s-1)}=\{\boldsymbol{y}_0, \boldsymbol{y}_1,\ldots, \boldsymbol{y}_{N-1}\} \subseteq \mathbb{R}^{s-1}$ generated by Algorithm \ref{algorithmRd} satisfies
\begin{equation*}
D_{N,\psi}^*(P_N^{(s-1)})
    \le 8LC^{-1} s b^{t/s}(2p-q) N^{-1/s},
\end{equation*}
for $N$ large enough, where $C=\int_{\mathbb{R}^{s-1}} \psi(\boldsymbol{z})d\boldsymbol{z}$ and $L$ is such that $\psi(\boldsymbol{x})\le LH(\boldsymbol{x})$ for all $\boldsymbol{x}\in \mathbb{R}^{s-1}$.
\end{theorem}

\begin{theorem}\label{lowerRealState} Let $H$ be a product density function defined on $\mathbb{R}^{s-1}$. Let $T$ be the transformation given in Equation~\eqref{InverseRosenblattTrans} associated  to $H$. Let $P_M$ be an arbitrary  point set in $[0,1]^s$, then $T(P_M)$ is a point set in $\mathbb{R}^{s-1}$. Then there exists an unnormalization density function $\psi$ defined in $\mathbb{R}^{s-1}$ satisfying the assumption in Theorem~\ref{GneralMaintheorem} such that  the star discrepancy of the points generated by the acceptance-rejection sampler with respect to $\psi$ and $H$ satisfies
\begin{equation*}
  D_{N,\psi}^*(P_N^{(s-1)})\ge c_sN^{-\frac{2}{s+1}},
\end{equation*}
where $c_s$ is independent of $N$ and $P_M$, but only dependent on $s$.
\end{theorem}

%{\begin{theorem}\label{lowerRealState} For an arbitrary $N$ point set $P_N\subset \mathbb{R}^s$, there exists a convex density function $\psi$ defined in $\mathbb{R}^{s-1}$ such that
%\begin{equation*}
% D_{N,\psi}^*(P_N)\ge c_sN^{-\frac{2}{s+1}}.
%\end{equation*}
%where $c_s >0$ is independent of $N$ and $P_N$, only depends on $s$.
%\end{theorem}}

\section{{Discrepancy properties of the deterministic reduced acceptance-rejection sampler}}\label{SecImprovedAlg}
Algorithm~\ref{Alg-RAR} can be extended to a  more general case.
Consider the target density $\psi(\boldsymbol{x})=\sum_{\ell=1}^{k} H_\ell(\boldsymbol{x}), \boldsymbol{x}\in D\subset \mathbb{R}^s$. If it is possible to  sample from $H_{\ell}(\boldsymbol{x})$ individually and the expectations of $H_{\ell}$  can be calculated or estimated with low cost, then we can use an embedding deterministic reduced acceptance-rejection sampler in each step.
{Let
\begin{align}\label{SjLj}
    &\mathcal{S}_{\ell}=\{\boldsymbol{x}\in D: \psi_{k-{\ell}+1}(\boldsymbol{x})< H_{\ell}(\boldsymbol{x})\},\nonumber \\
    \mbox{and}&\nonumber \\
    &\mathcal{L}_{\ell}=\{\boldsymbol{x}\in D: \psi_{k-{\ell}+1}(\boldsymbol{x})\ge H_{\ell}(\boldsymbol{x})\},\nonumber \\
\end{align} where $\psi_{k-{\ell}+1}(\boldsymbol{x})=\sum_{i={\ell}}^{k}H_i(\boldsymbol{x})$ for ${\ell}=1,\ldots,k-1$, and, in particular, $\psi_k$ is the target density.

Suppose we aim to sample $N$ points from the target density $\psi$. The sample set can be divided into two types, namely, points generated from the sets $\mathcal{S}_{\ell}$'s and $\mathcal{L}_{\ell}$'s respectively.
We apply a deterministic acceptance-rejection method given in Algorithm~\ref{algorithmRd} in each $\mathcal{S}_{\ell}$ with respect to $\psi_{k-{\ell}+1}$ and $H_{\ell}$. Note that we get  $\lceil N\int_{\mathcal{S}_{\ell}}\psi_{k-{\ell}+1}(\boldsymbol{x})d\boldsymbol{x}/\int_D\psi_k (\boldsymbol{x}) d\boldsymbol{x}\rceil$ points from $\mathcal{S}_{\ell}$ for ${\ell}=1,\ldots,k-1$.
For sampling from $\mathcal{L}_{\ell}$, the remaining samples come from applying the inverse transformation of $H_{\ell}$ in $\mathcal{L}_{\ell}$. Then we obtain additional   $\lceil N\int_{\mathcal{L}_{\ell}}H_{\ell}(\boldsymbol{x}) d\boldsymbol{x}/\int_D\psi_k(\boldsymbol{x}) d\boldsymbol{x}\rceil$ points from $\mathcal{L}_{\ell}$ for ${\ell}=1,\ldots,k$. We conduct the  procedure inductively until we get samples from $H_k(\boldsymbol{x})$. We assume that $\int_{\mathcal{S}_{\ell}}\psi_{k-{\ell}+1}(\boldsymbol{x})d\boldsymbol{x}/\int_D\psi_k (\boldsymbol{x}) d\boldsymbol{x}$ and $\int_{\mathcal{L}_{\ell}}H_{\ell}(\boldsymbol{x}) d\boldsymbol{x}/\int_D\psi_k(\boldsymbol{x}) d\boldsymbol{x}$ can be calculated or estimated.

The following algorithm is an extension of the  DRAR algorithm, which summarizes the embedding idea.
\begin{algorithm}\label{Alg-RAR-2} Let $\psi(\boldsymbol{x})=\sum_{{\ell}=1}^{k} H_{\ell}(\boldsymbol{x}), \boldsymbol{x}\in D\subset \mathbb{R}^{s-1}$, be a target density we aim to sample from. Define $\psi_{k-{\ell}+1}(\boldsymbol{x})=\sum_{i={\ell}}^{k}H_i(\boldsymbol{x})$ for $j=1,\ldots,k-1$. Denote $\mathcal{S}_{\ell}$ and $\mathcal{L}_{\ell}$ like in Equation~\eqref{SjLj} and assume that $\int_{\mathcal{S}_{\ell}}\psi_{k-{\ell}+1}(\boldsymbol{x})d\boldsymbol{x}/\int_D\psi (\boldsymbol{x}) d\boldsymbol{x}$ and $\int_{\mathcal{L}_{\ell}}H_{\ell}(\boldsymbol{x})d\boldsymbol{x}/\int_D\psi (\boldsymbol{x}) d\boldsymbol{x}$ can be calculated or estimated.
Further assume that we can sample from $H_{\ell}$ individually by applying the  transformation $T_{H_{\ell},\mathcal{S}_{\ell}}$ and $T_{H_{\ell},\mathcal{L}_{\ell}}$ given in Equation~\eqref{InverseRosenblattTrans} in $\mathcal{S}_{\ell}$ and $\mathcal{L}_{\ell}$ respectively. Suppose we aim to generate $N$ samples from $\psi$.  Let  $$N_{1,{\ell}}=\left\lceil N\frac{\int_{\mathcal{S}_{\ell}}\psi_{k-{\ell}+1}(\boldsymbol{x})d\boldsymbol{x}}{\int_D\psi (\boldsymbol{x}) d\boldsymbol{x} }\right\rceil \mbox{ and } N_{2,j}=\left\lceil N\frac{\int_{\mathcal{L}_{\ell}}H_{\ell}(\boldsymbol{x})d\boldsymbol{x}}{\int_D\psi (\boldsymbol{x}) d\boldsymbol{x}} \right\rceil.$$
For  ${\ell}$ from 1 to $k$ do:
\begin{itemize}
 \item [i)~]Let $\{\boldsymbol{x}_0,\boldsymbol{x}_1, \boldsymbol{x}_2,\ldots\}\subset [0,1]^s$ be a $(t,s)$-sequence in base $b$.

 %
% \item [ii)~] Choose $H_j$ as the proposal density and denote $\mathcal{S}_j$ and $\mathcal{L}_j$ like in Equation~\eqref{SjLj}. Assume
%      Let

 \item [ii)~] Compute $T_{H_{\ell},\mathcal{S}_{\ell}}(\boldsymbol{x}_n)$ for $n=0,1,2,\ldots$. Use the acceptance-rejection method with respect to  $\psi_{k-{\ell}+1}$ and  $H_{\ell}$ on the domain $\mathcal{S}_{\ell}$ using
 $\{\boldsymbol{x}_0,\boldsymbol{x}_1, \ldots,\boldsymbol{x}_{M-1}\}$ as driver sequence. Choose $M$ such that $N_{1,{\ell}}$ points are accepted by the DAR algorithm. Let $R_{1,{\ell}}=\{\boldsymbol{z}_0,\boldsymbol{z}_1, \ldots,\boldsymbol{z}_{N_{1,{\ell}}-1}\}$ be the accepted points.

 \item [iii)~]  Compute $T_{H_{\ell},\mathcal{L}_{\ell}}(\boldsymbol{x}_n)$ for $ n=0,1,\ldots,N_{2,{\ell}}-1$. Let $R_{2,{\ell}}=\{T_{H_{\ell},\mathcal{L}_{\ell}}(\boldsymbol{x}_n): 0\le n<N_{2,{\ell}}\}$.

 %Choose $M_{2,j}=b^{m_{2,j}}\ge\left\lceil N\int_{\mathcal{L}_j}H_j(\boldsymbol{x})d\boldsymbol{x}/\int_D\psi (\boldsymbol{x}) %d\boldsymbol{x} \right\rceil$, where $m_{2,j}\in\mathbb{N}$ is the smallest integer satisfying this inequality. Generate  a %$(t,m_{2,j},s)$-net $Q_{m_{2,j},s}$ in base $b$.  Use the inverse CDF of $H_j$ on $Q_{m_{2,j},s}$ to obtain  samples in $\mathcal{L}_j$.
  %Use inverse CDF of $H_j$ to obtain  approximately $N\int_{\mathcal{L}_j}H_j(\boldsymbol{x})d\boldsymbol{x}/\int_D\psi (\boldsymbol{x}) d\boldsymbol{x}$ points from $\mathcal{L}_j$.
  %\item [iv)~]For sampling from $\mathcal{L}_j$ with respect to $\psi_{k+j-2}$, we temporally update the target density to be $\psi_{k+j-2}$ and choose $H_{j+1}$ as a proposal density. Then denote $\mathcal{S}_{j+1}$ and $\mathcal{L}_{j+1}$.
% \item [v)~] If $j+1\le k$, turn to Step ii) and $iii)$; else stop.
 \end{itemize}
 Let $R_N^{(s)}=\bigcup_{{\ell}=1}^{k} (R_{1,{\ell}}\cup R_{2,{\ell}})$ and let $R_N^{(s-1)}$ denote the projection of $R_N^{(s)}$ onto the first $s-1$ coordinates. Return the set $R_N^{(s-1)}$.
\end{algorithm}
Now we consider the discrepancy properties of sample points produced by this algorithm.  Note that the sample set of $\psi=\sum_{{\ell}=1}^{k}H_{\ell}$ can be decomposed into several subsets with different star discrepancy.
\begin{theorem}\label{triangleDisc}
For a given target density $\psi(\boldsymbol{x})=\sum_{{\ell}=1}^{k} H_{\ell}(\boldsymbol{x}), \boldsymbol{x}\in D\subset \mathbb{R}^{s-1}$, let $\psi_{k-{\ell}+1}(\boldsymbol{x})=\sum_{i={\ell}}^{k}H_i(\boldsymbol{x})$.  Let $\mathcal{S}_{\ell}$ and $\mathcal{L}_{\ell}$ be given by \eqref{SjLj}.
Let $R^{(s-1)}_N$ be the sample set generated by Algorithm~\ref{Alg-RAR-2}, where $$N_{1,{\ell}}=\left\lceil N\int_{\mathcal{S}_{\ell}}\psi_{k-{\ell}+1}(\boldsymbol{x})d\boldsymbol{x}/\int_D\psi_k (\boldsymbol{x}) d\boldsymbol{x}\right\rceil,$$ which is the number of points generated from $\mathcal{S}_{\ell}$, and
$$N_{2,{\ell}}=\left\lceil N\int_{\mathcal{L}_{\ell}}H_{\ell}(\boldsymbol{x}) d\boldsymbol{x}/\int_D\psi_k(\boldsymbol{x}) d\boldsymbol{x}\right\rceil,$$ which is the number of points generated from $\mathcal{L}_{\ell}$ for ${\ell}=1,\ldots,k$. Assume that $N_{1,{\ell}}$ and $N_{2,{\ell}}$ can be calculated or estimated for the  given target density $\psi$ and $N$. Then we have
\begin{equation*}
    D_{N,\psi}^*(R^{(s-1)}_N)\le \sum_{{\ell}=1}^{k-1}\frac{N_{1,{\ell}}}{N}D^*_{\mathcal{S}_{\ell},\psi_{k-{\ell}+1}} +\sum_{{\ell}=1}^{k}\frac{N_{2,{\ell}}}{N}D^*_{\mathcal{L}_{\ell},H_{\ell}}+\frac{1}{N},
\end{equation*} where $D^*_{\mathcal{S}_{\ell},\psi_{k-{\ell}+1}}$ and $D^*_{\mathcal{L}_{\ell},H_{\ell}}$ is  the discrepancy of the samples in $\mathcal{S}_{\ell}$ and $\mathcal{L}_{\ell}$ respectively.
\end{theorem}
The proof of Theorem~\ref{triangleDisc} is given in Appendix~\ref{SumCaseDiscrepancy}. Note that this  method  achieves an improved acceptance rate of points since we are only  rejecting points in a certain range. For the remaining domain, we  get samples by applying the inverse transform. To be more exact, all point sets from $\mathcal{L}_{\ell}$ have low discrepancy since the inverse transformation is directly applied with respect to $H_{\ell}$ for ${\ell}=1,\ldots,k$. Now we consider the star discrepancy of points generated from $\mathcal{S}_{\ell}$.

The following result from \cite{Kr06} gives an improved upper bound on the star discrepancy on the first $M$ terms of a $(t,s)$-sequence in  base $b$ with $s\ge 2$.
\begin{lemma}\label{ImprovedMsequence}The star discrepancy of the first $M$ terms of a $(t,s)$-sequence $P_M$ in  base $b$ with $s\ge 2$ satisfies
\begin{equation*}
D_{M}^*(P_M) \le M^{-1} b^{t} (\log M)^{s} \frac{b^s(b-1)}{(b+1) 2^{s+1} (s)! (\log b)^{s}} + C_s M^{-1} b^{t} (\log M)^{s-1},
\end{equation*}
for some constant $C_s > 0$ only depending on $s$.
\end{lemma}
With the help of Lemma~\ref{StarIsotropic}, we obtain a bound on the isotropic discrepancy of the first $M$ points of a $(t,s)$-sequence.
\begin{lemma}Let the point set $P_M=\{\boldsymbol{x}_0,\boldsymbol{x}_1,\ldots,\boldsymbol{x}_{M-1}\}\subset[0,1]^s$ be the first $M$ terms of a $(t,s)$-sequence. For the isotropic discrepancy of $P_{M}$ we have
\begin{equation*}
  J_M(P_M)\le 2 s \left(\frac{4s}{s-1}\right)^{(s-1)/s} \Big( \frac{b^{t/s}  \big(\frac{b^s(b-1)}{(b+1) 2^{s+1} (s)! (\log b)^{s}} \big)^{1/s}\log M}{M^{1/s} } + \frac{C'_s  b^{t/s} (\log M)^{(s-1)/s}}{M^{1/s}} \Big),
\end{equation*} for some constant $C'_s>0$ depending only on $s$.
\end{lemma}

Hence, for the star discrepancy of $R_{1,{\ell}}$ for $1\le {\ell}< k$, using a $(t,s)$-sequence as a diver sequence in the DAR algorithm we have  a convergence rate of order $N_{1,{\ell}}^{-1/s}\log N_{1,{\ell}}$. We omit  a detailed proof since similar arguments as for proving Theorem~\ref{Maintheorem} can be used. The following corollary holds by substituting the proper upper bounds and $N_{1,{\ell}}, N_{2,{\ell}}$ in terms of $N$.
\begin{corollary} Suppose that  the target density $\psi(\boldsymbol{x})=\sum_{{\ell}=1}^{k} H_{\ell}(\boldsymbol{x}), \boldsymbol{x}\in D\subset \mathbb{R}^{s-1}$ satisfies all assumptions stated in  Theorem~\ref{triangleDisc}. Let $R^{(s-1)}_N$ be the sample set generated by Algorithm~\ref{Alg-RAR-2}. Then  we have
\begin{equation*}
   D_{N,\psi}^*(R^{(s-1)}_N)\le \sum_{j=1}^{k-1}\frac{C_{\mathcal{S}_{\ell},\psi_{k-{\ell}+1}}\alpha_j^{1-1/s}\log(\alpha_{\ell}N) }{N^{1/s}} +\sum_{{\ell}=1}^{k}\frac{C_{\mathcal{L}_{\ell},H_{\ell}}(\log\beta_{\ell}N)^{s-1}}{N}+\frac{1}{N},
\end{equation*} where
 $$\alpha_{\ell}=\frac{\int_{\mathcal{S}_{\ell}}\psi_{k-{\ell}+1}(\boldsymbol{x}) d\boldsymbol{x}}{\int_{D}\psi(\boldsymbol{x}) d\boldsymbol{x}}
  \quad \mbox{and} \quad \beta_{\ell}=\frac{\int_{\mathcal{L}_{\ell}}H_{\ell}(\boldsymbol{x}) d\boldsymbol{x}}{\int_{D}\psi(\boldsymbol{x}) d\boldsymbol{x}},$$
  and $C_{\mathcal{S}_{\ell},\psi_{k-{\ell}+1}}$ and $C_{\mathcal{L}_{\ell},H_{\ell}}$ are constants associated with $\mathcal{S}_{\ell},\psi_{k-{\ell}+1}$ and $\mathcal{L}_{\ell},H_{\ell}$ respectively.
\end{corollary}

\section{Conclusion and outlook}\label{Conclusion}
As is well known, the integration error using a Monte Carlo method converges at the rate of $N^{-1/2}$. The acceptance-rejection sampler with
 a deterministic driver sequence, which is a simple class of MCQMC methods,  performs much better in our numerical
 experiments than the theoretical result $N^{-1/s}$ of Theorem~\ref{Maintheorem} or Theorem~\ref{GneralMaintheorem},
  for a density defined on $[0,1]^{s-1}$ or $\mathbb{R}^{s-1}$, would imply. The three examples even demonstrate that it is
   possible to achieve a better convergence rate than  with standard  Monte Carlo using a well-chosen deterministic driver sequence.

The main drawback of the acceptance-rejection sampler is that it might reject many points in high dimension.
Some methods to improve the  acceptance rate of points are included in \cite{HLD04}. For a special class of density functions given by a finite sum, we propose an embedding deterministic reduced acceptance-rejection algorithm. This algorithm  produces a better star discrepancy convergence rate in our numerical example.
%For more general density functions,  an interesting topic to plug  deterministic driver sequences in those methods which have high acceptance rate. This is left for future work.

%\section*{Acknowledgements}
%H. Zhu was supported by a PhD scholarship from the University of New South Wales. J. Dick was supported by a Queen Elizabeth 2 Fellowship from the Australian Research Council. Helpful comments by Art Owen and Su Chen are gratefully acknowledged.

\newpage

\section*{Supplement}

\begin{appendix}

\section{$\delta$-cover to approximate star discrepancy}
Since it is computationally too expensive to compute the supremum in the definition of the star-discrepancy exactly for dimensions larger than one,  we use a so-called $\delta$-cover to estimate this supremum.

\begin{definition} Let $(G,\mathcal{B}(G),\psi)$ be a probability space where $G\subseteq \mathbb{R}^{s-1}$ and $\mathcal{B}(G)$ is the Borel $\sigma$-algebra defined on $G$.
Let $\mathscr{A} \subseteq\mathcal{B}(G)$ be a set of test sets.
A finite subset $\Gamma_\delta \subseteq\mathscr{A}$ is called a $\delta$-cover
of $\mathscr{A}$ with respect to $\psi$
 if for every $A \in \mathscr{A}$ there are sets $C, D \in \Gamma_\delta$ such that
\begin{equation*}
C \subseteq A \subseteq D
\end{equation*}
and
\begin{equation*}
\psi(D \setminus C) \le \delta.
\end{equation*}
\end{definition}

The concept of  $\delta$-cover is motivated by the following result \cite{Gnewuch2008}. Assume that $\Gamma_\delta$ is a $\delta$-cover of $\mathscr{A}$ with respect to the distribution $\psi$.
For all $\{\boldsymbol{z}_0,\boldsymbol{z}_1,\dots,\boldsymbol{z}_{N-1}\}$, the following discrepancy inequality holds
\begin{equation*}
\sup_{A \in \mathscr{A}} \left|\frac{1}{N} \sum_{n=0}^{N-1} 1_{\boldsymbol{z}_n \in A} - \psi(A) \right|
\le \max_{C \in \Gamma_\delta} \left|\frac{1}{N} \sum_{n=0}^{N-1} 1_{\boldsymbol{z}_n \in C} - \psi(C) \right| + \delta.
\end{equation*}

In the experiments we choose $\mathscr{A}$ to be the set of intervals $[\boldsymbol{0},\boldsymbol{t})$, where $\boldsymbol{t}$ runs through all points in the domain. For densities defined in $[0,1]^{s-1}$, we set  $\Gamma_\delta=\{\prod_{j=1}^{s-1}[0,a_j2^{-m}): a_j\in \mathbb{Z},0 \le a_j\le 2^m\}$,
   which means that the  $\delta$-cover becomes finer as the number of samples increases, thus it can yield a more accurate approximation of the star discrepancy.
For densities defined in $\mathbb{R}^{s-1}$, we choose $\delta$-covers with respect to $m$ as $\Gamma_\delta= \left\{\prod_{j=1}^{s-1}(0,F_j^{-1}(a_j 2^{-m})):a_j\in \mathbb{Z}, 0 \le a_j\le 2^m \right\},$ where $F_j^{-1}$ is the inverse marginal CDF with respect to the proposal density $H$.  Note that the approximation of the star-discrepancy is computationally expensive, thus our experiments only go up to several thousand sample points. However, the generation of samples using a $(t,m,s)$-net is fast.

\section{Proofs}
Before giving the proofs, we need some preparation.

Consider the following  elementary intervals 
\begin{equation}\label{ElementaryInterval}
W_k=\prod_{j=1}^{s}\left[\displaystyle\frac{c_j}{b^k},\displaystyle\frac{c_j+1}{b^k}\right),
\end{equation} \label{elementaryinterval}with $0\leq c_j< b^k$ (where $c_j$ is an integer) for $j=1,\ldots,s$.
The diagonal of $W_k$ has length $\sqrt s/b^k$ and the volume is $b^{-sk}$. Let  $J$ be an arbitrary convex set in $[0,1]^{s}$.
 Let $W_k^o$ denote the union of cubes $W_k$ fully contained in $J$,
\begin{equation}\label{OutW}W_k^o=\bigcup_{W_k\subseteq J} W_k.\end{equation}
Let $\overline{W}_k$ denote the union of cubes $W_k$ having non-empty intersection with $J$ or its boundary $\partial(J)$,
\begin{equation}\label{InW}\overline{W}_k=\bigcup_{W_k\cap (J\cup \partial(J))\neq \emptyset} W_k.\end{equation}

%The following result is \cite[Lemma]{NW75}.
\begin{lemma}\label{lemmaW} Let $k\in \mathbb{N}$. Let  $J$ be an arbitrary convex set in $[0,1]^{s}$. %Let a $(t,m,s)$-nets in base $b$ be  given.
 For the $W_k^o$ and $\overline{W}_k$ constructed by \eqref{OutW} and \eqref{InW}, we have
\begin{equation*}
\begin{array}{ll}
    \lambda(\overline{W}_k\setminus J)\leq 2s b^{-k} \mbox{ and }
    \lambda(J\setminus W_k^o )\leq 2s b^{-k}.
\end{array}
\end{equation*}
%where $\lambda(\cdot)$ is the Lebesgue measure.
\end{lemma}

To illustrate the result we provide the following simple argument which yields a slightly weaker result.
Based on the construction of $W_k$, its diagonal length is $\sqrt{s}/b^{k}$. Then
\begin{equation*}
    \overline{W}_k\setminus J\subseteq B:=\{\boldsymbol{x}\in [0,1]^s\setminus J:\parallel \boldsymbol{x}-\boldsymbol{y}\parallel \leq \sqrt sb^{-k}\mbox{ for some }\boldsymbol{y}\in J\},
\end{equation*}
where  $\parallel\cdot\parallel$ is the Euclidean norm.
Then
\begin{equation*}
    \lambda(\overline{W}_k\setminus J)\leq  \lambda(B).
\end{equation*}
Note that the outer surface area of a convex set in $[0,1]^{s}$ is bounded by the surface area of the unit cube $[0,1]^s$, which is $2s$. Thus the Lebesgue measure of the set $B$ is bounded by the outer surface area times the diameter.
 Therefore
 \begin{equation*}
   \lambda(\overline{W}_k\setminus J)\le \lambda(B) \le  2s\sqrt sb^{-k}.
\end{equation*}
The result for $\lambda(J\setminus W_k^o )$ follows by a  similar discussion as the proof above.

\begin{remark}
Note that in \cite{NW75} it was also shown that the constant $2s$ is best possible.
\end{remark}

Now we extend the result in Lemma \ref{lemmaW} to pseudo-convex sets.
\begin{corollary} \label{CorollaryWPseudo} Let  $J$ be an arbitrary pseudo-convex set in $[0,1]^{s}$ with admissible convex covering of $p$ parts with $q$ convex parts of $J$. %Let a $
For  $W^o_k$ and $\overline{W}_k$ given by \eqref{OutW} and \eqref{InW}  we have
\begin{equation*}
    \lambda(\overline{W}_k\setminus J)\leq 2psb^{-k} \mbox{ and }
    \lambda(J\setminus W_k^o )\leq 2psb^{-k}.
\end{equation*}
\end{corollary}

\begin{proof}
Let $A_1,\ldots,A_p$ be an admissible convex covering of $J$ with $p$ parts. Without loss of generality, let $A_1,\ldots,A_q$ be the convex subsets of $J$
and $A_{q+1},\ldots,A_p$ be such that  $A_j'=A_j\backslash J$ is convex for  $q+1\leq j\leq p$.
It can be shown that
\begin{equation}\label{SetA}
    J=\bigcup_{j=1}^q A_j \cup \bigcup_{j=q+1}^p ( A_j \backslash A_j').
\end{equation}

Therefore
\begin{equation*}
\overline{W}_k\setminus J \subseteq \left(\bigcup_{j=1}^{q} B_j \cup \bigcup_{j=q+1}^p B'_j\right),
\end{equation*}
where $$B_j=\left\{\boldsymbol{y}\in [0,1]^s\setminus A_j: \parallel \boldsymbol{x}-\boldsymbol{y}\parallel\le \sqrt s b^{-k} \mbox{ for some } \boldsymbol{x}\in A_j \right\},~j=1,\ldots,q,$$
and
  $$B'_j=\{\boldsymbol{y}\in A'_j: \parallel \boldsymbol{x}-\boldsymbol{y}\parallel\le \sqrt s b^{-k} \mbox{ for some } \boldsymbol{x}\in [0,1]^s\setminus A_j'\},~j=q+1,\ldots,p.$$
Since $B_j\cup A_j$ for $j=1,\ldots, q$, and $B'_j\cup A'_j$ for $j=q+1,\ldots,p$ are convex,
using Lemma \ref{lemmaW}, we obtain
\begin{equation*}
\lambda(\overline{W}_k\setminus J) \leq  \lambda\Big(\bigcup_{j=1}^{q} B_j\Big) + \lambda\Big(\bigcup_{j=q+1}^p B'_j\Big)
\leq 2psb^{-k}.
\end{equation*}

The result for  $\lambda(J\setminus W_k^o )$ follows by a similar discussion.
\end{proof}

%%\begin{proof}
%%We now use (\ref{SetA}) and apply it to the local discrepancy function.
%%Thus
%%\begin{equation*}
%%    \begin{array}{ll}
%%    &\displaystyle\frac{1}{N}\displaystyle\sum_{n=0}^{N-1}1_A(\boldsymbol{x}_n)-\lambda(A)\\
%%    =&\displaystyle\sum_{j=0}^{q}\left[ \displaystyle\frac{1}{N}\sum_{n=0}^{N-1}1_{A_j}(\boldsymbol{x}_n)-\lambda(A_j)\right]+\sum_{j=q+1}^{p}\left[\displaystyle\frac{1}{N}\sum_{n=0}^{N-1}1_{A_j\backslash A_j'}(\boldsymbol{x}_n)-\lambda(A_j\backslash A_j')\right]\\
%%%    =& \displaystyle\sum_{j=0}^{q}\left[\displaystyle\frac{1}{N}\sum_{n=0}^{N-1}1_{A_j}(\boldsymbol{x}_n)-\lambda(A_j)\right]+\sum_{j=q+1}^{p}\left[ \displaystyle\frac{1}{N}\sum_{n=0}^{N-1}1_{A_j}(\boldsymbol{x}_n)-\lambda(A_j)\right]\\
%%    \quad &-\displaystyle\sum_{j=q+1}^{p}\left[ \displaystyle\frac{1}{N}\sum_{n=0}^{N-1}1_{A_j'}(\boldsymbol{x}_n)-\lambda(A_j')\right].
%%    \end{array}
%%\end{equation*}
%%By the definition of the isotropic  discrepancy and the fact that $A_j$, for $j=1,\ldots,q$, and $A_j'$, for $j=q+1,\ldots, p$, are convex sets, we obtain
%%\begin{equation*}\begin{array}{ll}
%%   \left|\displaystyle\frac{1}{N}\sum_{n=0}^{N-1}1_A(\boldsymbol{x}_n)-\lambda(A)\right|
%%    &\leq [q+(p-q)+(p-q)] J_N(P_N)\\
%%    &=(2p-q)J_N(P_N).
%%\end{array}\end{equation*}
%%\end{proof}
%

\subsection{Proof of upper bound}\label{AppendixUpperBound}
%Proof of Lemma~\ref{lemmaConvex}.

%Proof of Corollary~\ref{CorollaryWPseudo}.

Proof of Lemma~\ref{IsoDisBound}.
\begin{proof}
For the point set $Q_{m,s}=\{\boldsymbol{x}_0,\boldsymbol{x}_1,\ldots, \boldsymbol{x}_{M-1}\}\subseteq[0,1]^s$ generated by a $(t,m,s)$-net in base $b$ with $M=b^m$,   let $k=\left\lfloor \displaystyle\frac{m-t}{s}\right\rfloor$. Let  $J$ be an arbitrary convex set in $[0,1]^{s}$.
Consider the elementary interval $W_k$ given by Equation~\eqref{ElementaryInterval}. For  $W^o_k$ and $\overline{W}_k$ given by \eqref{OutW} and \eqref{InW},
%%\begin{equation}
%%W_k=\prod_{i=1}^{s}\left[\displaystyle\frac{c_i}{b^k},\displaystyle\frac{c_i+1}{b^k}\right),
%%\end{equation}
%%with $0\leq c_i< b^k$ (where $c_i$ is an integer) for $i=1,\ldots, s$.
%%The diagonal of $W_k$ has length $\sqrt s/b^k$ and the volume is $b^{-sk}$.
%
% Let $J$ be an arbitrary convex subset of $A$ and let $W_k^o$ denote the union of cubes $W_k$ fully contained in $J$,
% $$W_k^o=\bigcup_{W_k\subseteq J}W_k.$$
%Let  $\overline{W}_k$ denote the union of cubes $W_k$ having non-empty intersection with $J$ or its boundary,
%   $$\overline{W}_k=\bigcup_{W_k\cap (J\cup \partial(J))\neq \emptyset} W_k.$$
obviously, $W_k^o\subseteq J\subseteq \overline{W}_k$. The sets $W_k^o$ and $\overline{W}_k$
 are fair with respect to the net, that is
 \begin{align*}
    \displaystyle\frac{1}{M} \sum_{n=0}^{M-1}1_{\overline{W}_k}(\boldsymbol{x}_n)=\lambda(\overline{W}_k)\quad\mbox{and}\quad
           \displaystyle\frac{1}{M}  \sum_{n=0}^{M-1}1_{W_k^o}(\boldsymbol{x}_n)=\lambda(W_k^o).
 \end{align*}
 Then
\begin{align*}
\frac{1}{M} \sum_{n=0}^{M-1} 1_{J}(\boldsymbol{x}_n)-\lambda(J)  \leq & \frac{1}{M} \sum_{n=0}^{M-1} 1_{\overline{W}_k}(\boldsymbol{x}_n)-\lambda(\overline{W}_k)+\lambda(\overline{W}_k \setminus J)\\ = & \lambda(\overline{W}_k \setminus J),
\end{align*}
 and
\begin{align*}
\frac{1}{M}\sum_{n=0}^{M-1}1_{J}(\boldsymbol{x}_n)-\lambda(J) \geq &
\frac{1}{M}\sum_{n=0}^{M-1}1_{W_k^o}(\boldsymbol{x}_n)-\lambda(W^o_k)-\lambda( J\setminus W_k^o) \\
     = & -\lambda( J\setminus W_k^o).
 \end{align*}
  By Lemma \ref{lemmaW}, we have
 \begin{equation*}
   \lambda(\overline{W}_k\setminus J)\leq 2sb^{-k}  \mbox{ and }
    \lambda(J\setminus W^o_k )\leq2sb^{-k}.
\end{equation*}
Thus we obtain
\begin{equation*}
    \abs{\frac{1}{M}\sum_{n=0}^{M-1}1_{J}(\boldsymbol{x}_n)-\lambda(J)}\leq2sb^{-k}\leq 2sb^{t/s}M^{-1/s}.
\end{equation*}
Since the bound holds for arbitrary convex sets, the proof is complete.
\end{proof}

Proof of Theorem~\ref{Maintheorem}.

\begin{proof}
 Let $J_{\boldsymbol{t}}^*=([\boldsymbol{0},\boldsymbol{t})\times [0,1])\bigcap A $, where $\boldsymbol{t}=(t_1,\ldots,t_{s-1})$ and $A=\{\boldsymbol{x}\in[0,1]^s:\psi(x_1,,\ldots,x_{s-1})\ge Lx_s\}$.
 Since $\boldsymbol{y}_n $ are the first $s-1$ coordinates  of $\boldsymbol{z}_n \in A$ for $ n=0,\ldots,N-1$, we have $$\sum_{n=0}^{M-1}1_{J_{\boldsymbol{t}}^*}(\boldsymbol{x}_n)=\sum_{n=0}^{N-1}1_{J_{\boldsymbol{t}}^*}(\boldsymbol{z}_n)
 =\sum_{n=0}^{N-1}1_{[\boldsymbol{0},\boldsymbol{t})}(\boldsymbol{y}_n).$$
 Therefore
 \begin{equation*}
   \left|\displaystyle\frac{1}{N}\sum_{n=0}^{N-1}1_{[\boldsymbol{0},\boldsymbol{t})}({\boldsymbol{y}_n})
    -\displaystyle\frac{1}{C}\int_{[\boldsymbol{0},\boldsymbol{t})}\psi({\boldsymbol{z}})d\boldsymbol{z}\right|=\left|\displaystyle\frac{1}{N}\sum_{n=0}^{M-1}1_{J_{\boldsymbol{t}}^*}({\boldsymbol{x}_n})
    - \frac{1}{\lambda(A)} \lambda(J_{\boldsymbol{t}}^*) \right|.
 \end{equation*}
The right-hand side above is now bounded by
\begin{eqnarray*}
&& \frac{M}{N} \left|\frac{1}{M} \sum_{n=0}^{M-1} 1_{J_{\boldsymbol{t}}^*}(\bsx_n) - \lambda(J_{\boldsymbol{t}}^*) \right| + \left| \lambda(J_{\boldsymbol{t}}^*) \left(\frac{M}{N} - \frac{1}{\lambda(A)} \right)\right| \\ & \le & \frac{M}{N} \left(\left|\frac{1}{M} \sum_{n=0}^{M-1} 1_{J_{\boldsymbol{t}}^*}(\bsx_n) - \lambda(J_{\boldsymbol{t}}^*) \right| + \left| \lambda(A) - \frac{1}{M} \sum_{n=0}^{M-1} 1_{A}(\bsx_n) \right| \right),
\end{eqnarray*}
where we used the estimation $\lambda(J^*_{\boldsymbol{t}}) \le \lambda(A)$ and the fact that $N = \sum_{n=0}^{M-1} 1_{A}(\bsx_n)$. Since $J^*_{\boldsymbol{t}}$ is also pseudo-convex, it follows from Lemma~\ref{lemmaConvex} that we can bound the above expression by
\begin{eqnarray*}
&& %\frac{M}{N} (2p-q) J_M(P_M^{(s)}) + C \left|\frac{M}{N} - \frac{1}{C} \right| \\ & \le &  \frac{M}{N} \left( (2p-q) J_M(P_M^{(s)}) +  \left|C - \frac{1}{M} \sum_{n=0}^{M-1} 1_{A}(\bsx_n) \right| \right) \\ & \le &
\frac{M}{N} 2 (2p-q) J_M(P_M^{(s)}).
\end{eqnarray*}

In addition, $\lim_{M\to\infty} \frac{N}{M}=\lambda(A)$,  which means  $\lim_{M\to\infty} \frac{N}{M}=\int_{[0,1]^{s-1}}\psi(\boldsymbol{z})d\boldsymbol{z}/L=C/L $. Hence there is an $M_0$ such that $\frac{N}{M} \geq C/(2L)$ for all $M \ge M_0$. Thus $\frac{M}{N} \le \frac{2L}{C}$ for all $M \ge M_0$. Further we have $N \le M$. Using Lemma~\ref{IsoDisBound} we obtain the bound
\begin{equation*}
 \frac{M}{N} 2 (2p-q) J_M(P_M^{(s)}) \le 8L C^{-1}  s b^{t/s} (2p-q) N^{-1/s}.
\end{equation*}
\end{proof}

%%%%%%%%%%
%\begin{remark} We call a nonnegative function pseudo-convex if and only if  the region below its graph is a pseudo-convex  set.\end{remark}

%The following lemma will be used to get the discrepancy bound  of a point set on a pseudo-convex set. It is an extension of \cite[Lemma %5]{CJJ2011} to the unit cube.
%\begin{lemma} \label{lemmaConvex}
%Let $A$ be a pseudo-convex subset of $[0,1]^s$ with admissible convex covering of $p$ parts with $q$ convex parts of $A$. Then for any point %set $P_N=\{\boldsymbol{x}_0,\ldots,\boldsymbol{x}_{N-1}\}\subseteq [0,1]^s$ we have
%\begin{equation*}
%    \left|\displaystyle\frac{1}{N}\sum_{n=0}^{N-1}1_A(\boldsymbol{x}_n)-\lambda(A)\right|\leq (2p-q) J_N(P_N).
%\end{equation*}
%\end{lemma}
\subsection{Proof of lower bound}\label{AppendixLowerbound}
The following lemma  provides information about the packing number of the northern hemisphere
$$\mathbb{S}_{\tiny{\mbox{north}}}^{s-1}:=\{\boldsymbol{x}\in [0,1]^{s}, \parallel \boldsymbol{x}-\boldsymbol{1/2} \parallel=1/2, x_{s}\ge 0\},$$
where $\boldsymbol{1/2}=(\frac{1}{2},\frac{1}{2},\ldots,\frac{1}{2})\in [0,1]^{s}$. The (closed) spherical cap $C(\boldsymbol{y},\rho)\subset \mathbb{S}_{\tiny{\mbox{north}}}^{s-1}$ with center $\boldsymbol{y}\in \mathbb{S}_{\tiny{\mbox{north}}}^{s-1}$ and angular radius $\rho\in [0,\pi]$ is defined by
\begin{equation*}
   C(\boldsymbol{y},\rho)=\{\boldsymbol{y}\in \mathbb{S}_{\tiny{\mbox{north}}}^{s-1}|\boldsymbol{x}\cdot\boldsymbol{y}\ge \frac{\cos \rho}{4}\}.
\end{equation*}
%To construct a density function we use spherical caps which we introduce in the following.
%\begin{definition}(Spherical cap) Let $\mathbb{S}^s=\{\boldsymbol{x}\in \mathbb{R}^{s+1}| \parallel \boldsymbol{x}\parallel=1\}$ be the unit sphere in $\mathbb{R}^{s+1}$. The (closed) spherical cap $C(\boldsymbol{y},\rho)\subset \mathbb{S}^s$ with center $\boldsymbol{y}\in \mathbb{S}^s$ and angular radius $\rho\in [0,\pi]$ is defined by
%\begin{equation*}
%    C(\boldsymbol{y},\rho)=\{\boldsymbol{y}\in \mathbb{S}^s|\boldsymbol{x}\cdot\boldsymbol{y}\ge \cos \rho ).
%\end{equation*}
%\end{definition}
The packing of $\mathbb{S}_{\tiny{\mbox{north}}}^{s-1}$ considered here is constructed by identical spherical caps which are non-overlapping, that is, $C(\boldsymbol{y}_i,\rho)$ and $C(\boldsymbol{y}_j,\rho)$ with $i\neq j$ touch at most at their boundaries.

\begin{lemma}\label{Coveringhemisphere} Let $s\ge 1$. For any $n\in\mathbb{N}$ there exist $M_n$ points $\boldsymbol{y}_1,\ldots,\boldsymbol{y}_{M_n}$ on the northern  hemisphere $\mathbb{S}_{\tiny{\mbox{north}}}^s\subset[0,1]^{s+1}$ and an angular radius $\rho_n$, with
\begin{align*}
    \rho_n=c_1(2n)^{-1/(s-1)},&\\
    n\le M_n\le c_2 n,&
\end{align*}
such that the caps $ C(\boldsymbol{y}_i,\rho_n)$, $i=1,\ldots,M_n$, form a packing of the northern hemisphere. The positive constant $c_1,c_2$ depend only on the dimension $s$.
\end{lemma}

The lemma is essentially well-known for spheres.  The explicit proof is due to Wyner~\cite{Wyner1965} and Hesse gives a summary in \cite[Lemma~1]{Hesse2006}. A similar argument can be used for the hemisphere in our case.

Now we give the proof of Theorem~\ref{lowerPsiN} whose proof follows the argument from the proof of  \cite[Theorem~1]{Schmidt1975}.

\begin{proof} We may suppose $s\ge2$. Let $\mathbb{S}_{\tiny{\mbox{north}}}^{s-1}$ be the northern hemisphere defined above contained in $[0,1]^s$, and let $S$ be the surface of  $\mathbb{S}_{\tiny{\mbox{north}}}^{s-1}$. Let $C$ be a closed spherical cap on $S$ with spherical radius $\rho$.  The convex hull $\overline{C}$ of $C$ is a solid spherical cap. For $0<\rho<\pi/2$, $\lambda(\overline{C})$ is a continuous function of  $\rho$ with
\begin{equation}\label{lambda_C}
    c_1\rho^{s+1} < \lambda(\overline{C}) < c_2\rho^{s+1}
\end{equation}
If $N$ is sufficiently large, there is a positive real number of $\rho_0$ such that a cap $C$ of spherical radius $\rho_0$ has
\begin{equation*}
    \lambda(\overline{C})=\frac{1}{2N}.
\end{equation*}
In view of \eqref{lambda_C}, $0<\rho_0<c_3N^{-1/(s+1)}$. We now pick as many pairwise disjoint caps with radius $\rho_0$ as possible, say $C_1,\ldots,C_M$. By Lemma~\ref{Coveringhemisphere}, for large  $N$ and hence small $\rho_0$ we have $M\ge c_4 \rho_0^{-(s-1)}$, hence
\begin{equation}\label{value_M}
M\ge c_5N^{(s-1)/(s+1)}.
\end{equation}

Given a sequence of numbers $\sigma_1,\ldots,\sigma_M$, with each $\sigma_i$ either $1$ or $-1$, let $B(\sigma_1,\ldots,\sigma_M)$ consist of all $x\in \mathbb{S}_{\tiny{\mbox{north}}}^{s-1}$ which do not  lie in a cap $\overline{C_i}$ with $\sigma_i=-1$. In other words, $B(\sigma_1,\ldots,\sigma_M)$ is obtained form $\mathbb{S}_{\tiny{\mbox{north}}}^{s-1}$ by removing the solid caps $\overline{C_i}$ for which $\sigma_i=-1$.

Now the local discrepancy function $\Delta_{P_N}(H)$ defined by $\Delta_{P_N}(H)=\sum_{i=1}^N 1_{H}(P_N)-N\lambda(H)$ is additive, i.e. it satisfies
\begin{equation*}
    \Delta_{P_N}(H\cup H')=\Delta_{P_N}(H)+\Delta_{P_N}(H')
\end{equation*}
if $H\cap H'=\emptyset$. It follows easily that
\begin{equation*}
   \Delta_{P_N}(B(\sigma_1,\ldots,\sigma_M))-\Delta_{P_N}(B(-\sigma_1,\ldots,-\sigma_M))=\sum_{i=1}^M\sigma_i \Delta_{P_N}(\overline{C_i}).
\end{equation*}

We have
\begin{equation*}
    \Delta_{P_N}(\overline{C_i})=\sum_{i=1}^N 1_{\overline{C_i}}(P_N)-N\lambda(\overline{C_i})=\sum_{i=1}^N 1_{\overline{C_i}}(P_N)-\frac{1}{2}.
\end{equation*}
Hence for every $i$, either $\Delta_{P_N}(\overline{C_i})\ge \frac{1}{2}$ or $\Delta_{P_N}(\overline{C_i})=-\frac{1}{2}$. Choose $\sigma_i$ such that $\sigma_i \Delta_{P_N}(\overline{C_i})\ge 1/2~(1\le i\le M)$. Then
\begin{equation*}
 \Delta_{P_N}(B(\sigma_1,\ldots,\sigma_M))-\Delta_{P_N}(B(-\sigma_1,\ldots,-\sigma_M))\ge M/2,
\end{equation*}
and either $J=B(\sigma_1,\ldots,\sigma_M)$ or $J=B(-\sigma_1,\ldots,-\sigma_M)$ has $|\Delta_{P_N}(J)|\ge M/4$. In addition, $J$ is a convex set due to its construction.

Thus by \eqref{value_M},
\begin{equation*}
    D_N^*(\lambda,J)\ge \frac{1}{4}\frac{M}{N}\ge c_6 N^{-2/(s+1)}.
\end{equation*}
We take $\psi$ as the boundary  of $J$ excluding the boundary of $[0,1]^s$, which completes the proof.
\end{proof}
\subsection{Proof of Theorem~\ref{triangleDisc}}\label{SumCaseDiscrepancy}
\begin{proof} In what follows we restrict our investigations to the case $k=2$ for simplicity, the general case can be proved by similar arguments. Let $\psi=H_1+H_2$ be the target density function. Assume that we can apply the inverse CDF on $H_1$ and $H_2$ to generate samples. Let  \begin{align*}
    &\mathcal{S}:=\{\boldsymbol{x}\in D: \psi(\boldsymbol{x})< H_1(\boldsymbol{x})\},\nonumber \\
    \mbox{and}&\nonumber \\
    &\mathcal{L}:=\{\boldsymbol{x}\in D: \psi(\boldsymbol{x})\ge H_1(\boldsymbol{x})\}.\nonumber \\
\end{align*}
 The final sample set $R^{(1)}_N$ is a superposition of the three subsets, $R_{1,1},R_{1,2}$ and $R_{2,2}$, see Figure~\ref{FishProblem}. Define $R_{i,\ell}=\{x_0^{(i,\ell)},x_1^{(i,\ell)},\ldots,x_{N_{i,\ell}-1}^{(i,\ell)}\}$ for $i,\ell=1,2$. The number $N_{i,\ell}$ of the points in each subset is given by
 \begin{equation*}
   N_{1,1}=\left\lceil N\frac{\int_{\mathcal{S}}\psi(\boldsymbol{x})d\boldsymbol{x}}{\int_D\psi(\boldsymbol{x})d\boldsymbol{x}}\right\rceil,
   N_{1,2}=\left\lceil N\frac{\int_{\mathcal{L}}H_1(\boldsymbol{x})d\boldsymbol{x}}{\int_D\psi(\boldsymbol{x})d\boldsymbol{x}}\right\rceil \mbox{ and }  \,
   N_{2,2}=\left\lceil N\frac{\int_{\mathcal{L}}H_2(\boldsymbol{x})d\boldsymbol{x}}{\int_D\psi(\boldsymbol{x})d\boldsymbol{x}}\right\rceil.
 \end{equation*}
Then there exists  $\delta_i\in[0,1)$ for $i=1,2,3$ such that  $$N_{1,1}= N\frac{\int_{\mathcal{S}}\psi(\boldsymbol{x})d\boldsymbol{x}}{\int_D\psi(\boldsymbol{x})d\boldsymbol{x}} +\delta_1, N_{1,2}= N\frac{\int_{\mathcal{L}}H_1(\boldsymbol{x})d\boldsymbol{x}}{\int_D\psi(\boldsymbol{x})d\boldsymbol{x}}+ \delta_2 \mbox{ and } N_{2,2}= N\frac{\int_{\mathcal{L}}H_2(\boldsymbol{x})d\boldsymbol{x}}{\int_D\psi(\boldsymbol{x})d\boldsymbol{x}}+\delta_3.$$

 Therefore,
 \begin{align*}
  &\left|\frac{1}{N}\sum_{n=0}^{N-1}1_{(-\infty,\boldsymbol{t}]\cap D}(\boldsymbol{x}_n)-\frac{\int_{(-\infty,\boldsymbol{t}]\cap D}\psi(\boldsymbol{x})d\boldsymbol{x}}
  {\int_D\psi(\boldsymbol{x})d\boldsymbol{x}}\right|\\
 =&\left|\frac{1}{N}\sum_{n=0}^{N-1}1_{(-\infty,\boldsymbol{t}]\cap D}(\boldsymbol{x}_n)
 -\frac{\int_{(-\infty,\boldsymbol{t}]\cap\mathcal{L}}H_1(\boldsymbol{x})d\boldsymbol{x}+\int_{(-\infty,\boldsymbol{t}]\cap\mathcal{L}}H_2(\boldsymbol{x})d\boldsymbol{x}
 +\int_{(-\infty,\boldsymbol{t}]\cap\mathcal{S}}\psi(\boldsymbol{x})d\boldsymbol{x}}
  {\int_D\psi(\boldsymbol{x})d\boldsymbol{x}}\right|\\
 \le &\left|\frac{1}{N}\sum_{n=0}^{N_{1,1}-1}1_{(-\infty,\boldsymbol{t}]\cap D}(\boldsymbol{x}_n^{(1,1)})
 -\frac{\int_{(-\infty,\boldsymbol{t}]\cap\mathcal{L}}H_1(\boldsymbol{x})d\boldsymbol{x}}{\int_D\psi(\boldsymbol{x})d\boldsymbol{x}}\right|\\
 &+\left|\frac{1}{N}\sum_{n=0}^{N_{1,2}-1}1_{(-\infty,\boldsymbol{t}]\cap D}(\boldsymbol{x}_n^{(1,2)})
 -\frac{\int_{(-\infty,\boldsymbol{t}]\cap\mathcal{L}}H_2(\boldsymbol{x})d\boldsymbol{x}}{\int_D\psi(\boldsymbol{x})d\boldsymbol{x}}\right|\\
 &+\left|\frac{1}{N}\sum_{n=0}^{N_{2,2}-1}1_{(-\infty,\boldsymbol{t}]\cap D}(\boldsymbol{x}_n^{(2,2)})
 -\frac{\int_{(-\infty,\boldsymbol{t}]\cap\mathcal{S}}\psi(\boldsymbol{x})d\boldsymbol{x}}{\int_D\psi(\boldsymbol{x})d\boldsymbol{x}}\right|\\
 = &\frac{N_{1,1}}{N}\left|\frac{1}{N_{1,1}}\sum_{n=0}^{N_{1,1}-1}1_{(-\infty,\boldsymbol{t}]\cap D}(\boldsymbol{x}^{(1,1)}_n)
-\frac{N}{N_{1,1}}\frac{\int_{(-\infty,\boldsymbol{t}]\cap\mathcal{S}}\psi(\boldsymbol{x})d\boldsymbol{x}}{\int_D\psi(\boldsymbol{x})d\boldsymbol{x}}\right|\\ &+\frac{N_{1,2}}{N}\left|\frac{1}{N_{1,2}}\sum_{n=0}^{N_{1,2}-1}1_{(-\infty,\boldsymbol{t}]\cap D}(\boldsymbol{x}^{(1,2)}_n)
-\frac{N}{N_{1,2}}\frac{\int_{(-\infty,\boldsymbol{t}]\cap\mathcal{L}}H_1(\boldsymbol{x})d\boldsymbol{x}}{\int_D\psi(\boldsymbol{x})d\boldsymbol{x}}\right|\\
& +\frac{N_{2,2}}{N}\left|\frac{1}{N_{2,2}}\sum_{n=0}^{N_{2,2}-1}1_{(-\infty,\boldsymbol{t}]\cap D}(\boldsymbol{x}^{(2,2)}_n)
-\frac{N}{N_{2,2}}\frac{\int_{(-\infty,\boldsymbol{t}]\cap\mathcal{L}}H_2(\boldsymbol{x})d\boldsymbol{x}}{\int_D\psi(\boldsymbol{x})d\boldsymbol{x}}\right|\\
=&\frac{N_{1,1}}{N}\left[\left|\frac{1}{N_{1,1}}\sum_{n=0}^{N_{1,1}-1}1_{(-\infty,\boldsymbol{t}]\cap D}(\boldsymbol{x}^{(1,1)}_n)
-\frac{\int_{(-\infty,\boldsymbol{t}]\cap\mathcal{S}}\psi(\boldsymbol{x})d\boldsymbol{x}}{\int_{\mathcal{S}}
\psi(\boldsymbol{x})d\boldsymbol{x}}\right|+\frac{\delta_1}{N_{1,1}}\frac{\int_{(-\infty,\boldsymbol{t}]\cap\mathcal{S}}\psi(\boldsymbol{x})d\boldsymbol{x}}{\int_{\mathcal{S}}
\psi(\boldsymbol{x})d\boldsymbol{x}}\right]\\
&+\frac{N_{1,2}}{N}\left[\left|\frac{1}{N_2}\sum_{n=0}^{N_{1,2}-1}1_{(-\infty,\boldsymbol{t}]\cap D}(\boldsymbol{x}^{(1,2)}_n)
-\frac{\int_{(-\infty,\boldsymbol{t}]\cap\mathcal{L}}H_1(\boldsymbol{x})d\boldsymbol{x}}{\int_{\mathcal{L}}H_1(\boldsymbol{x})d\boldsymbol{x}}\right|
+\frac{\delta_2}{N_{1,2}}\frac{\int_{(-\infty,\boldsymbol{t}]\cap\mathcal{L}}H_1(\boldsymbol{x})d\boldsymbol{x}}{\int_{\mathcal{L}}H_1(\boldsymbol{x})d\boldsymbol{x}}\right]\\
&+\frac{N_{2,2}}{N}\left[\left|\frac{1}{N_{2,2}}\sum_{n=0}^{N_{2,2}-1}1_{(-\infty,\boldsymbol{t}]\cap D}(\boldsymbol{x}^{(2,2)}_n)
-\frac{\int_{(-\infty,\boldsymbol{t}]\cap\mathcal{L}}H_2(\boldsymbol{x})d\boldsymbol{x}}{\int_{\mathcal{L}}H_2(\boldsymbol{x})d\boldsymbol{x}}\right|
+\frac{\delta_3}{N_{2,2}}\frac{\int_{(-\infty,\boldsymbol{t}]\cap\mathcal{L}}H_2(\boldsymbol{x})d\boldsymbol{x}}{\int_{\mathcal{L}}H_2(\boldsymbol{x})d\boldsymbol{x}}\right]\\ \le &\frac{N_{1,1}}{N} D^*_{\mathcal{S},\psi}
+ \frac{N_{1,2}}{N} D^*_{\mathcal{L},H_1}
+\frac{N_{2,2}}{N} D^*_{\mathcal{L},H_2}+\frac{1}{N}\frac{\int_{(-\infty,\boldsymbol{t}]\cap D}\psi(\boldsymbol{x})d\boldsymbol{x}}
  {\int_D\psi(\boldsymbol{x})d\boldsymbol{x}}\\
\le &\frac{N_{1,1}}{N} D^*_{\mathcal{S},\psi}
+ \frac{N_{1,2}}{N} D^*_{\mathcal{L},H_1}
+\frac{N_{2,2}}{N} D^*_{\mathcal{L},H_2}+\frac{1}{N},\\
\end{align*}
where $D^*_{\mathcal{S},\psi}$ is the star discrepancy of sample points in $\mathcal{S}$ associated  with $\psi$ and the same notation is also applied to $D^*_{\mathcal{L},H_1}$ and $D^*_{\mathcal{L},H_2}$. Since this result holds for arbitrary $\boldsymbol{t}$, the desired result follows then immediately.
\end{proof}
\end{appendix}
\end{document}